\documentclass[11 pt]{article}
\usepackage{amsmath, amsthm,amssymb,enumerate,mathrsfs}
\usepackage{tikz,xcolor}
\usetikzlibrary{calc,shapes}
\usepackage{hyperref}
\usepackage{a4wide}
\usepackage{cite}

\usepackage[multiple]{footmisc}
\makeatletter
\renewcommand*\@fnsymbol[1]{\the#1}  
\makeatother

\title{Supersaturation in Posets and Applications Involving the Container Method}
\author{Jonathan A. Noel\thanks{Department of Mathematics, ETH, Raemistrasse 101, HG G 33.4, 8092 Zurich, Switzerland. E-mail: \texttt{\{jonathan.noel,benjamin.sudakov\}@math.ethz.ch}.} $^{\text{, }}$\thanks{Part of this work was completed while the first author was a DPhil student at the University of Oxford.}
\and Alex Scott\thanks{Mathematical Institute, University of Oxford, Andrew Wiles Building, Radcliffe Observatory Quarter, Woodstock Road, Oxford, OX2 6GG, UK. E-mail: \texttt{scott@maths.ox.ac.uk}.}
\and Benny Sudakov\footnotemark[1]
}

\newtheoremstyle{case}{}{}{\normalfont}{}{\itshape}{:}{ }{}

\newtheorem{thm}[equation]{Theorem}
\newtheorem{lem}[equation]{Lemma}
\newtheorem{prop}[equation]{Proposition}
\newtheorem{conj}[equation]{Conjecture}
\newtheorem{cor}[equation]{Corollary}
\newtheorem{claim}[equation]{Claim}

\theoremstyle{definition}
\newtheorem{defn}[equation]{Definition}

\newtheorem{rem}[equation]{Remark}

\newtheorem*{noteAdded}{Note added in the proof}

\newtheoremstyle{case}{}{}{\normalfont}{}{\itshape}{\normalfont:}{ }{}

\theoremstyle{case}
\newtheorem{case}{Case}
\newtheorem{case2}{Case}

\newtheorem{case4}{Case}

\numberwithin{equation}{section}

\DeclareMathOperator{\comp}{comp}

\newcommand{\qbinom}[3]{\left[\genfrac{}{}{0pt}{}{#1}{#2}\right]_{#3}}
\newcommand{\Vectors}[2]{\mathcal{V}\left(#1,#2\right)}

\newcommand\numberthis{\addtocounter{equation}{1}\tag{\theequation}}

\date{}

\begin{document}

\maketitle

\begin{abstract}
We consider `supersaturation' problems  in partially ordered sets (posets) of the following form. Given a finite poset $P$ and an integer $m$ greater than the cardinality of the largest antichain in $P$, what is the minimum number of comparable pairs in a subset of $P$ of cardinality $m$? We provide a framework for obtaining lower bounds on this quantity based on counting comparable pairs relative to a random chain and apply this framework to obtain supersaturation results for three classical posets: the boolean lattice, the collection of subspaces of $\mathbb{F}_q^n$ ordered by set inclusion and the set of divisors of the square of a square-free integer under the `divides' relation. The bound that we obtain for the boolean lattice can be viewed as an approximate version of a known theorem of Kleitman~\cite{superKleitman}.

In addition, we apply our supersaturation results to obtain (a) upper bounds on the number of antichains in these posets and (b) asymptotic bounds on the cardinality of the largest  antichain in $p$-random subsets of these posets which hold with high probability (for $p$ in a certain range). The proofs of these results rely on a  `container-type' lemma for posets which generalises a result of Balogh, Mycroft and Treglown~\cite{RandomSperner}. We also state a number of open problems regarding supersaturation in posets and counting antichains.
\end{abstract}

\section{Introduction}

A large part of extremal combinatorics is concerned with determining the maximum size of a combinatorial object subject to the constraint that it does not contain a certain `forbidden' substructure. One example of such an extremal problem is to determine the cardinality of the largest subset of a partially ordered set (poset) $P$ which does not contain a comparable pair; i.e. to determine the cardinality of the largest \emph{antichain}. Perhaps the most famous result in this area is Sperner's Theorem~\cite{Sperner} which says that the size of the largest antichain in the \emph{boolean lattice} poset $\mathcal{P}(n)$, consisting of all subsets of $[n]:=\{1,\dots,n\}$ ordered by inclusion, is $\binom{n}{\left\lfloor n/2\right\rfloor}$. 

A closely related problem is to determine the minimum number of copies of a particular substructure in a combinatorial object of prescribed size. These so called `supersaturation' problems have a long history starting with a result of Rademacher (unpublished, see~\cite{Rademacher}) which says  that every graph with $n$ vertices and more than $n^2/4$ edges must contain at least $\left\lfloor n/2\right\rfloor$ triangles. For additional background on supersaturation problems in extremal combinatorics see, e.g.,~\cite{Reiher,Patkos,additive, DGS2, SS2} and the references therein.

Our focus in this paper is on a natural supersaturation problem for posets: \emph{given a finite poset $P$ and a positive integer $m\leq |P|$, what is the minimum number of comparable pairs in a subset $S$ of $P$ of cardinality $m$?}  An early result in this direction is due to Kleitman~\cite{superKleitman} who proved that, for $1\leq m\leq 2^n$, the number of comparable pairs in a subset of $\mathcal{P}(n)$ of cardinality $m$ is minimised by a collection of $m$ subsets of $[n]$ whose cardinalities are as close to $n/2$ as possible. This result strengthens Sperner's Theorem and settled a conjecture of Erd\H{o}s and Katona (see~\cite{DGKS,DGS, BaloghWagnerKleitman} for further results). 

In this paper, we prove new supersaturation results for two classical posets (Theorems~\ref{vecSpThm} and~\ref{multisetThm} below). Our method involves counting comparisons relative to a random chain and should be applicable to posets beyond those covered in this paper. We also illustrate the method by using it to prove an approximate version of the supersaturation theorem of Kleitman~\cite{superKleitman} in the boolean lattice (Theorem~\ref{booleanThm}). All three of these supersaturation results will be applied to prove bounds on the number of antichains in these posets and bounds on the size of antichains in random subsets of these posets.

Throughout the paper, for a subset $S$ of a poset $P$, we let $\comp(S)$ denote the number of comparable pairs in $S$.

\begin{defn}
Given a prime power $q$ and integers $n \geq 1$  let $\Vectors{q}{n}$ be the poset of all subspaces of $\mathbb{F}_q^n$ ordered by inclusion. 
For an integer $0\leq i\leq n$, the \emph{$q$-binomial coefficient} $\qbinom{n}{i}{q}$ is the number of subspaces of $\mathbb{F}_q^n$ of dimension $i$.
\end{defn}

Elementary counting shows that 
\[\qbinom{n}{i}{q}=\prod_{j=0}^{i-1}\frac{1-q^{n-j}}{1-q^{j+1}}.\]
One may also observe that, for fixed $n$ and $q$, the sequence $\qbinom{n}{0}{q},\dots,\qbinom{n}{n}{q}$ is unimodal and satisfies $\qbinom{n}{i}{q} = \qbinom{n}{n-i}{q}$ for all $0\leq i\leq n$. We prove the following theorem.

\begin{thm}
\label{vecSpThm}
Let $q$ be a prime power and $k$ be a fixed positive integer. There exists a constant $n_0(k)$ such that if $n\geq n_0(k)$ and $S\subseteq \Vectors{q}{n}$ has cardinality at least 
\[\sum_{r=0}^{k-1}\qbinom{n}{\left\lceil\frac{n-k+1+2r}{2}\right\rceil}{q}+t,\] then
\[\comp(S)\geq t\qbinom{\left\lceil (n+k)/2\right\rceil}{k}{q}.\]
\end{thm}

It is easy to see that Theorem~\ref{vecSpThm} is best possible when $k=1$ and $0\leq t \leq \qbinom{n}{\left\lfloor (n-1)/2\right\rfloor}{q}$ as the set $S$ consisting of all subspaces of dimension $\left\lceil n/2\right\rceil$ and $t$ subspaces of dimension $\left\lfloor (n-1)/2\right\rfloor$ contains exactly $ t\qbinom{\left\lceil (n+1)/2\right\rceil}{1}{q}$ comparable pairs. 

\begin{defn}
Given $n\geq 1$, we define a poset on $\{0,1,2\}^n$ where $x\leq y$ if and only if $x_j\leq y_j$ for $1\leq j\leq n$, where $x_j$ and $y_j$ are the $j$th coordinates of $x$ and $y$, respectively. 
For $0\leq i\leq 2n$, let $\ell_i(n)$ denote the number of vectors in $\{0,1,2\}^n$ whose coordinates sum to $i$.
\end{defn}

The poset $\{0,1,2\}^n$ can also be viewed as the poset of divisors of the square of a square-free integer with $n$ prime factors under the `divides' relation or the poset of all submultisets of a multiset consisting of $n$ distinct elements each of multiplicity two under inclusion. It is easy to see that
\[\ell_i(n) = \sum_{s=\max\{0,i-n\}}^{\left\lfloor i/2\right\rfloor}\binom{n}{i-s}\binom{i-s}{s}.\] Also, trivially, $\ell_i(n)=\ell_{2n-i}(n)$ for $0\leq i\leq 2n$. A result of de Bruijn, van Ebbenhorst Tengbergen and Kruyswijk~\cite{multiSperner} says that $\{0,1,2\}^n$ has a `symmetric chain decomposition' (see \cite{AndersonBook} for details) which, in particular, implies that the sequence $\ell_0(n),\dots,\ell_{2n}(n)$ is unimodal. We prove the following theorem.

\begin{thm}
\label{multisetThm}
Let $k$ be a fixed positive integer. Then there exists a constant $n_0(k)$ such that if $n\geq n_0(k)$ and $S\subseteq \{0,1,2\}^n$ has cardinality at least 
\[\sum_{r =n-\left\lfloor \frac{k-1}{2}\right\rfloor}^{n+\left\lceil \frac{k-1}{2}\right\rceil}\ell_{r}(n)+ t,\] 
then
\[\comp(S)\geq \left(\frac{\ell_{3k-1}(n)}{\ell_{2k-1}(n)} - 1\right)t.\]
\end{thm}

Note that for $k=1$ we have $\frac{\ell_{3k-1}(n)}{\ell_{2k-1}(n)} - 1 = \frac{n-1}{2}$ and, for general $k$ and $n$ large with respect to $k$, we have 
\begin{equation}\label{ratioAsymp}\frac{\ell_{3k-1}(n)}{\ell_{2k-1}(n)} - 1 = \frac{(1+o(1)\binom{n}{3k-1}}{\binom{n}{2k-1}} = \frac{(1+o(1))(2k-1)!n^k}{(3k-1)!}.\end{equation}
In particular, for $k=1$ and $|S|$ not much larger than $\ell_n(n)$, Theorem~\ref{multisetThm} is nearly best possible. To see this, let $S = S_n\cup S_{n+1}$ where $S_n$ contains all elements of $\{0,1,2\}^n$ whose coordinates sum to $n$ and $S_{n+1}$ contains the $|S|-\ell_n(n)$ elements of $\{0,1,2\}^n$ whose coordinates sum to $n+1$ and, subject to this, have the minimum number of non-zero coordinates. 

In the past few years, there has been some considerable success in extending classical results in extremal combinatorics to sparse random settings due to the breakthroughs 
made independently by Conlon and Gowers~\cite{ConlonGowers} and Schacht~\cite{Schacht}. Shortly thereafter another, very different, approach known as the `container method' 
was developed independently by Balogh, Morris and Samotij~\cite{ContainersMorris} and Saxton and Thomason~\cite{ContainersThomason}. As was demonstrated 
in~\cite{ContainersThomason,ContainersMorris}, the container method can be used to prove many of the results of~\cite{ConlonGowers,Schacht}. Two advantages of the container 
method are that it is relatively straightforward to apply and that it often provides `counting' results which can be much stronger than the corresponding probabilistic 
statements. Supersaturation theorems form a key ingredient in most proofs involving the container method. For instance, many of the proofs 
in~\cite{ContainersThomason,ContainersMorris} applied the classical graph supersaturation theorem of Erd\H{o}s and Simonovits~\cite{Hyper}. For additional applications of the 
container method, see~\cite{BaloghWagnerBoolean,GraphCont,HongNew} and the references therein.

We apply Theorems~\ref{vecSpThm} and~\ref{multisetThm} and the container method to prove upper bounds on the number of antichains in $\Vectors{q}{n}$ and $\{0,1,2\}^n$ (Theorem~\ref{countVec} and~\ref{countMulti}) and upper bounds on the largest antichains in random subsets of these posets of a certain density which hold with high probability (Theorems~\ref{randVecSp} and~\ref{randMultiset}).

\begin{thm}
\label{countVec}
Let $q$ be a fixed prime power. The number of antichains in $\mathcal{V}(q,n)$ is $2^{\left(1+O\left(\sqrt{n}q^{-n/4}\right)\right)\qbinom{n}{\left\lfloor n/2\right\rfloor}{q}}$. 
\end{thm}

\begin{thm}
\label{countMulti}
The number of antichains in $\{0,1,2\}^n$ is $2^{\left(1+O\left(\sqrt{\log{n}/n}\right)\right)\ell_n(n)}$. 
\end{thm}

Of course, the number of antichains in any poset $P$ is at least $2$ to the power of the cardinality of its largest antichain, and so Theorems~\ref{countVec} and~\ref{countMulti} are best possible up to the $(1+o(1))$ factor in the exponent.

Given a set $X$, a \emph{$p$-random subset} of $X$ is a subset of $X$ obtained by including each element of $X$ with probability $p$ independently of one another.

\begin{thm}
\label{randVecSp}
Let $q$ be a fixed prime power and $\varepsilon>0$. There exists a positive constant $c(\varepsilon,q)$ such that if $p\geq c(\varepsilon,q)/q^{n/2}$, then with high probability every antichain in a $p$-random subset of $\Vectors{q}{n}$ has cardinality at most $(1+\varepsilon)p\qbinom{n}{\left\lfloor n/2\right\rfloor}{q}$. 
\end{thm}

\begin{thm}
\label{randMultiset}
Let $\varepsilon>0$ be fixed. There exists a positive constant $c(\varepsilon)$ such that if $p\geq c(\varepsilon)/n$, then with high probability every antichain in a $p$-random subset of $\{0,1,2\}^n$ has cardinality at most $(1+\varepsilon)p\ell_n(n)$. 
\end{thm}

It is not hard to see that Theorems~\ref{randVecSp} and~\ref{randMultiset} are best possible up to the choice of the constants $c(\varepsilon,q)$ and $c(\varepsilon)$; this is made explicit in Section~\ref{countAnti}.

In the next section, we prove a general lemma for obtaining lower bounds on the number of comparable pairs in a subset $S$ of a poset $P$. We provide a few basic applications of this lemma in Section~\ref{exampleSec} including a proof of an approximate version of the theorem of Kleitman~\cite{superKleitman} and a proof of Theorem~\ref{vecSpThm}. In Section~\ref{multisetSec}, we apply this lemma again to prove Theorem~\ref{multisetThm}. In Section~\ref{containerSec} we prove a container-type lemma for posets which generalises a result of~\cite{RandomSperner} proved for the boolean lattice. In Section~\ref{countAnti}, we use this lemma to prove Theorems~\ref{countVec},~\ref{countMulti},~\ref{randVecSp} and~\ref{randMultiset} and a known upper bound on the number of antichains in the boolean lattice included for the sake of demonstration. Throughout the paper, we will require several known bounds on $\qbinom{n}{i}{q}$ and $\ell_i(n)$ which are stated in the appendix.

\section{Comparison Counting Via Random Chains}
\label{randomChainSec}

Our general approach to proving supersaturation results, which was inspired by  \cite{DGS}, is to count comparable pairs in a subset $S$ of $P$ by considering the expected number of such pairs in a chain of $P$ chosen randomly according to some distribution. 
One should point out that the general idea of counting comparable pairs relative to a random chain goes back, in some sense, to Lubell's proof~\cite{LYML} of the LYM inequality. It was recently used
by Das, Gan and Sudakov~\cite{DGS} and Dove, Griggs, Kang and Sereni~\cite{DGKS} to count chains of length $k$ in subsets of $\mathcal{P}(n)$ of prescribed cardinality.
In order for this technique to be effective, one requires a distribution on the  chains in $P$ which is sufficiently `balanced.' Roughly speaking, what we mean by this is that a random  chain chosen according to this distribution is not too unlikely to contain any fixed element of $P$ and, at the same time, not too likely to contain any fixed comparable pair of $P$. Our estimates will use a lemma on random chains, for which we require the notion of a comparability digraph.

\begin{defn}
Let $P$ be a poset. A \emph{comparability digraph} for $P$ is a digraph $D$ with $V(D)=P$ such that
\begin{itemize}
\item for every directed arc $(x,y)$ of $D$, we have that $x$ is comparable to $y$ in $P$, and
\item for every comparable pair $x,y\in P$, exactly one of the arcs $(x,y)$ or $(y,x)$ is in $E(D)$. 
\end{itemize}
\end{defn}

We remark that a poset $P$ can have many comparability digraphs. In particular, the fact that a comparability digraph contains an arc from $x$ to $y$ does not determine which of the relations $x<y$ or $y<x$ holds. 

\begin{lem}
\label{randomCount}
Let $P$ be a poset, let $D$ be a comparability digraph for $P$, let $\mu$ be a distribution on the chains of $P$ and let $\mathcal{C}_\mu$ be a random chain chosen according to $\mu$. Suppose further that $\mathbb{P}\left(x\in \mathcal{C}_\mu\right)>0$ for every $x\in P$. Then, for every set $S\subseteq P$, we have
\[\comp(S)\geq \left({\max_{(x,y)\in E(D)}\mathbb{P}\left(y\in \mathcal{C}_\mu\mid x\in\mathcal{C}_\mu\right)}\right)^{-1}\left(|S| - \left({\min_{x\in P}\mathbb{P}\left(x\in\mathcal{C}_\mu\right)}\right)^{-1}\right).\]
\end{lem}

\begin{proof}
We proceed by induction on $|S|$ where the base case $|S|\leq \left({\min_{x\in P}\mathbb{P}\left(x\in\mathcal{C}_\mu\right)}\right)^{-1}$ is trivial. Given $w\in S$, we define $d_S^+(w)$ to be the number of arcs $(w,z)\in E(D)$ with $z\in S$. Suppose that there exists $w\in S$ with $d^+_S(w)\geq \left({\max_{(x,y)\in E(D)}\mathbb{P}\left(y\in \mathcal{C}_\mu\mid x\in\mathcal{C}_\mu\right)}\right)^{-1}$. In this case, we apply the inductive hypothesis to $S\setminus\{w\}$ to obtain at least 
\[\left({\max_{(x,y)\in E(D)}\mathbb{P}\left(y\in \mathcal{C}_\mu\mid x\in\mathcal{C}_\mu\right)}\right)^{-1}\left(|S\setminus \{w\}|- \left({\min_{x\in P}\mathbb{P}\left(x\in\mathcal{C}_\mu\right)}\right)^{-1}\right).\]
comparable pairs in $S$ which do not involve $w$. We then obtain the desired bound on $\comp(S)$ by considering these comparable pairs along with those of the form $(w,z)\in E(D)$ with $z\in S$. Thus, we may assume that every $w\in S$ satisfies 
\begin{equation}\label{outdeg}d_S^+(w)< \left({\max_{(x,y)\in E(D)}\mathbb{P}\left(y\in \mathcal{C}_\mu\mid x\in\mathcal{C}_\mu\right)}\right)^{-1}.\end{equation}
The chain $\mathcal{C}_\mu$ contains precisely $|\mathcal{C}_\mu\cap S|$ elements of $S$ and  $\binom{|\mathcal{C}_\mu\cap S|}{2}$ of the comparable pairs of $S$ (since every pair in $\mathcal{C}_\mu$ is comparable). Since $k-\binom{k}{2}\leq 1$ for all $k$, we get that
\[1\geq \mathbb{E}\left(|\mathcal{C}_\mu\cap S| - \binom{|\mathcal{C}_\mu\cap S|}{2}\right) =\sum_{w\in S}\mathbb{P}\left(w\in\mathcal{C}_\mu\right) - \sum_{\substack{w,z\in S \\ (w,z)\in E(D)}}\mathbb{P}\left(w,z\in \mathcal{C}_\mu\right)\]
\[=\sum_{w\in S}\mathbb{P}\left(w\in \mathcal{C}_\mu\right)\left(1-\sum_{\substack{z\in S\\ (w,z)\in E(D)}}\mathbb{P}\left(z\in \mathcal{C}_\mu\mid w\in \mathcal{C}_\mu\right)\right)\]
\[\geq\sum_{w\in S}\mathbb{P}\left(w\in \mathcal{C}_\mu\right)\left(1-d^+_S(w){\max_{(x,y)\in E(D)}\mathbb{P}\left(y\in \mathcal{C}_\mu\mid x\in\mathcal{C}_\mu\right)}\right).\]
By \eqref{outdeg}, all of the summands are positive, and so this expression is bounded below by
\[\sum_{w\in S}\left({\min_{x\in P}\mathbb{P}\left(x\in\mathcal{C}_\mu\right)}\right)\left(1-d^+_S(w){\max_{(x,y)\in E(D)}\mathbb{P}\left(y\in \mathcal{C}_\mu\mid x\in\mathcal{C}_\mu\right)}\right).\]
Rearranging, we obtain
\[\comp(S) = \sum_{w\in S}d_S^+(w) \geq \left({\max_{(x,y)\in E(D)}\mathbb{P}\left(y\in \mathcal{C}_\mu\mid x\in\mathcal{C}_\mu\right)}\right)^{-1}\left(|S| - \left({\min_{x\in P}\mathbb{P}\left(x\in\mathcal{C}_\mu\right)}\right)^{-1}\right)\]
as desired.
\end{proof}

\begin{rem}
We remark that the proof of this lemma  generalises to give a lower bound on the number of edges induced by a set of $m$ vertices in a general graph $G$ in terms of $m$. That is, one could orient the edges of the graph and select a random clique according to a probability distribution $\mu$.
\end{rem}

\section{Supersaturation in \texorpdfstring{$\boldsymbol{\mathcal{P}(n)}$}{P(n)} and \texorpdfstring{$\boldsymbol{\Vectors{q}{n}}$}{V(q,n)}}
\label{exampleSec}

To illustrate the applicability of Lemma~\ref{randomCount}, we prove a weak version of the theorem of  Kleitman~\cite{superKleitman} in the boolean lattice which is often good enough for applications involving the container method (see, e.g.,~\cite{RandomSperner,RandomErdos,BaloghWagnerBoolean} and Theorem~\ref{countBoolean}).

\begin{thm}[Kleitman~\cite{superKleitman}]
\label{booleanThm}
Let $k$ be a fixed positive integer. Then there exists a constant $n_0(k)$ such that if $n\geq n_0(k)$ and $S\subseteq \mathcal{P}(n)$ is a set of cardinality at least 
\[\sum_{r=0}^{k-1}\binom{n}{\left\lceil \frac{n-k+1 +2r}{2}\right\rceil} + t,\]
then
\[\comp(S)\geq t\binom{\left\lceil (n+k)/2\right\rceil}{k}.\]
\end{thm}

\begin{proof}
For $0\leq r\leq k-1$, let $P_r$ be the poset consisting all subsets of $[n]$ of cardinality equivalent to $\left\lceil \frac{n-k+1+2r}{2}\right\rceil$ modulo $k$ ordered by inclusion (the idea of splitting the levels of $\mathcal{P}(n)$ into equivalence classes modulo $k$ is also used by Collares Neto and Morris~\cite{RandomErdos}). Note that $P_0\cup \dots \cup P_{k-1}$ partitions $\mathcal{P}(n)$. Define $S_r:=S\cap P_{r}$ and
\[t_r:=|S_r| - \binom{n}{\left\lceil \frac{n-k+1+2r}{2}\right\rceil}.\] 
Clearly, $\sum_{r=0}^{k-1}t_r = t$. Our goal is to show that, for $0\leq r\leq k-1$, the following holds:
\begin{equation}\label{theEllCase}\comp(S_r)\geq t_r\binom{\left\lceil (n+k)/2\right\rceil}{k}.\end{equation} 
This will imply the result since  $\comp(S)\geq \sum_{r=0}^{k-1}\comp(S_r)$.

Let $r\in \{0,\dots,k-1\}$ be fixed. We prove \eqref{theEllCase} by induction on $t_r$. Note that \eqref{theEllCase} holds trivially if $t_r\leq 0$. Now, for $t_r\geq1$, if $S_r$ contains an element $x$ which is comparable to at least $\binom{\left\lceil (n+k)/2\right\rceil}{k}$ elements of $S_r$, then we apply the induction hypothesis to find at least $(t_r-1)\binom{\left\lceil (n+k)/2\right\rceil}{k}$ comparable pairs in $S_r\setminus\{x\}$ and we are done.  Therefore, we can assume that there is no such $x\in S_r$.

Next, we show that we can assume that $S$ does not contain any extremely small sets. Let $s\in\{0,\dots,k-1\}$ so that $\left\lceil \frac{n-k+1+2r}{2}\right\rceil\equiv s\bmod k$. Let $X$ be the set of all elements of $S_r$ of cardinality $s$. Suppose that $X\neq\emptyset$ and for each $x\in X$ let $N_{k+s}(x)$ be the elements of $P_r\setminus S_r$ of cardinality $k+s$ which are comparable to $x$. By the result of the previous paragraph, $x$ is comparable to fewer than $\binom{\left\lceil (n+k)/2\right\rceil}{k}$ elements of $S_r$ and so 
\[|N_{k+s}(x)|> \binom{n-s}{k} - \binom{\left\lceil (n+k)/2\right\rceil}{k} \geq \binom{n-k+1}{k}-\binom{\left\lceil (n+k)/2\right\rceil}{k}\]
which, if $n_0(k)$ is chosen large enough with respect to $k$, is at least
\[\binom{n}{k-1} \geq \binom{n}{s} \geq |X|.\]
Therefore, we can greedily associate each element $x$ of $X$ to an element $y_x$ of $P_r\setminus S_r$ of cardinality $k+s$ such that $x$ is comparable to $y_x$ and the $y_x$ are distinct from one another. It is clear that the set  $S_r':=(S_r\setminus X)\cup \{y_x: x\in X\}$ has the same cardinality as $S_r$ and at most as many comparable pairs as $S_r$. So, in what follows, we can assume that $X=\emptyset$; that is, every element of $S_r$ has cardinality at least $k+s$.

Now, define $P_r^*$ to be the poset obtained from $P_r$ by deleting all sets of cardinality $s$. Let $D$ be a digraph with $V(D)=P_r^*$ containing every arc $(x,y)$ where either
\begin{itemize}
\item $x$ and $y$ are comparable and the cardinality of $y$ is closer to $n/2$ than the cardinality of $x$, or
\item  $x<y$ and $|x| = n-|y|$.
\end{itemize}
Let $\mathcal{C}_\mu$ be a random chain in $P_r^*$ obtained by taking a maximal chain in $\mathcal{P}(n)$ uniformly at random and intersecting it with $P_r^*$. Then any two elements of $P_r^*$ of the same cardinality are equally likely to be contained in $\mathcal{C}_\mu$ and so
\[\mathbb{P}\left(x\in \mathcal{C}_\mu\right)\geq \binom{n}{\left\lceil \frac{n-k+1+2r}{2}\right\rceil}^{-1}\text{ for all $x\in P_r^*$}.\]
Also, given $(x,y)\in E(D)$, we have
\[\mathbb{P}\left(y\in \mathcal{C}_\mu\mid x\in\mathcal{C}_\mu\right) =\begin{cases}
\binom{|x|}{|y|}^{-1} & \text{if $x>y$},\\
\binom{n-|x|}{|y|-|x|}^{-1}& \text{otherwise}.
  \end{cases}
\]
By definition of $D$ and the fact that $X=\emptyset$, we get that this quantity is always at most $\binom{\left\lceil (n+k)/2\right\rceil}{k}^{-1}$. Thus, \eqref{theEllCase} holds by Lemma~\ref{randomCount}.
\end{proof}

The proof of Theorem~\ref{booleanThm} above exploits the following two properties of a uniformly random maximal chain $\mathcal{C}_\mu$ in the boolean lattice:
\begin{itemize}
\item  any two subsets of $[n]$ of the same cardinality are equally likely to be contained in $\mathcal{C}_\mu$ and
\item for any two comparable pairs $(x_1,y_1)$ and $(x_2,y_2)$ in $\mathcal{P}(n)$ such that $|x_1|=|x_2|$ and $|y_1|=|y_2|$, we have $\mathbb{P}\left(x_1,y_1\in\mathcal{C}_\mu\right) = \mathbb{P}\left(x_2,y_2\in\mathcal{C}_\mu\right)$.
\end{itemize}
It is clear that one can often obtain a supersaturation result for posets satisfying these two conditions by taking a maximal chain uniformly at random and applying Lemma~\ref{randomCount}. We demonstrate this by outlining a proof of Theorem~\ref{vecSpThm}.

\begin{proof}[Proof of Theorem~\ref{vecSpThm}]
For $0\leq r\leq k-1$, let $P_r$ be the poset consisting all subspaces of $\mathbb{F}_q^n$ of dimension equivalent to $\left\lceil \frac{n-k+1+2r}{2}\right\rceil$ modulo $k$ ordered by inclusion. Analogous to the proof of Theorem~\ref{booleanThm}, we define $S_r:=S\cap P_{r}$ and
\[t_r:=|S_r| - \qbinom{n}{\left\lceil \frac{n-k+1+2r}{2}\right\rceil}{q}.\] 
Clearly, $\sum_{r=0}^{k-1}t_r = t$. Our goal is to prove that, for each $r=0,\dots,k-1$,
\begin{equation}\label{theEllCaseVec}\comp(S_r)\geq t_r\qbinom{\left\lceil (n+k)/2\right\rceil}{k}{q}.\end{equation} 

We prove \eqref{theEllCaseVec} by induction on $t_r$ where the case $t_r\leq 0$ is trivial. By the inductive hypothesis, we can assume that every element of $S_r$ is comparable to fewer than $\qbinom{\left\lceil (n+k)/2\right\rceil}{k}{q}$ other elements of $S_r$. If $S_r$ contains an element of dimension $s<k$, then we can shift all such elements up to subspaces of dimension $k+s$ without increasing the number of comparable pairs (as in the proof of Theorem~\ref{booleanThm}) since
\[\qbinom{n-k+1}{k}{q} - \qbinom{\left\lceil (n+k)/2\right\rceil}{k}{q} \geq \qbinom{n}{k-1}{q}\]
if $n_0(k)$ is chosen large enough with respect to $k$. 

Now, we let $P_r^*$ be the poset obtained from $P_r$ by deleting all elements of dimension less than $k$ and let $D$ be a digraph with $V(D)=P_r^*$ containing every arc $(x,y)$ such that either
\begin{itemize}
\item $x$ is comparable to $y$ and the dimension of $y$ is closer to $n/2$ than the dimension of $x$, or
\item $x<y$ and the dimension of $x$ is $n$ minus the dimension of $y$. 
\end{itemize}
Let $\mathcal{C}_\mu$ be a uniformly random chain in $P_r^*$ obtained by taking a random maximal chain in $\Vectors{q}{n}$ and intersecting it with $P_r^*$. It is not hard to show that
\[\mathbb{P}\left(x\in \mathcal{C}_\mu\right)\geq \qbinom{n}{\left\lceil \frac{n-k+1+2r}{2}\right\rceil}{q}^{-1}\text{ for all $x\in P_r^*$ and}\]
\[\mathbb{P}\left(y\in\mathcal{C}_\mu\mid x\in\mathcal{C}_\mu\right)\leq \qbinom{\left\lceil (n+k)/2\right\rceil}{k}{q}^{-1}\text{ for all $(x,y)\in E(D)$.}\]
The result follows.
\end{proof}

Our main difficulty in proving Theorem~\ref{multisetThm} is that one can no longer obtain a good bound by taking a uniformly random maximal chain in $\{0,1,2\}^n$ and applying Lemma~\ref{randomCount}. Much of the next section is devoted to constructing a more appropriate distribution on chains in $\{0,1,2\}^n$ and studying its properties.

\section{Supersaturation in \texorpdfstring{$\boldsymbol{\{0,1,2\}^n}$}{\{0,1,2\}\textasciicircum n}}
\label{multisetSec}

We open this section with a basic definition.

\begin{defn}
A poset $P$ is said to be \emph{ranked} if it can be partitioned into non-empty sets $L_0,\dots,L_N$ called \emph{levels} such that 
\begin{itemize}
\item if $x\in L_i$ and $x<y$, then $y\in \bigcup_{j>i} L_j$,
\item if $x<y$ such that $x\in L_i$ and $y\in L_j$ for $j>i+1$, then there exists $z\in L_{i+1}$ such that $x<z<y$. 
\end{itemize}
If $x\in L_i$, then we say that the \emph{rank} of $x$ is $i$.
\end{defn}

It is clear that $\{0,1,2\}^n$ is a ranked poset where $L_i$ is the set of vectors in $\{0,1,2\}^n$ whose coordinates sum to $i$. Also, the posets $\mathcal{P}(n)$ and $\Vectors{q}{n}$ are both ranked posets where the rank of an element is equal to its cardinality or dimension, respectively.

In the previous section, we used the fact that a uniformly random chain of maximal length in $\mathcal{P}(n)$ or $\Vectors{q}{n}$ is equally likely to contain any element of rank $i$. The poset $\{0,1,2\}^n$ does not have this property. That is, one can easily show that for $n\geq2$ a uniformly random maximal chain is not equally likely to contain any given element in $L_i$ for $2\leq i\leq 2n-2$. Therefore, if we are to apply Lemma~\ref{randomCount} to obtain a supersaturation result which applies to all subsets of $\{0,1,2\}^n$ of cardinality greater than $\ell_n(n)$, then we will require a different distribution on the chains of $\{0,1,2\}^n$. 

\begin{defn}
Let $P$ be a ranked poset with non-empty levels $L_0,\dots,L_N$. A non-empty finite multiset $\mathscr{C}$ of maximal chains in $P$ is called a \emph{regular covering by chains} if for  $0\leq i\leq N$ every element of $L_i$ is contained in exactly $|\mathscr{C}|/|L_i|$ chains of $\mathscr{C}$. 
\end{defn}

As noted above,  the collection of all maximal chains in $\{0,1,2\}^n$ is not a regular covering by chains.  However, Anderson~\cite{Andersondivisors} proved that $\{0,1,2\}^n$ does possess a regular covering by chains (in fact, he proved that the same is true for the poset of divisors of any integer; see also~\cite{AndersonBook}). The following useful result of Kleitman~\cite{Kleitman} characterises ranked posets with regular coverings by chains in terms of two other natural conditions, one of which will be used in our proof. 

\begin{defn}
Let $P$ be a ranked poset with levels $L_0,\dots,L_N$. Then $P$ is said to have the \emph{LYM property} if, for every antichain $A\subseteq P$,
\[\sum_{i=0}^N\frac{|A\cap L_i|}{|L_i|}\leq 1.\]
\end{defn}

\begin{defn}
Let $P$ be a ranked poset with levels $L_0,\dots,L_N$. Then $P$ is said to have the \emph{normalised matching property} if for $0\leq i\leq N-1$ and $T\subseteq L_i$ we have
\[\frac{|\{y\in L_{i+1}: y>x \text{ for some }x\in T\}|}{|L_{i+1}|}\geq \frac{|T|}{|L_i|}.\]
\end{defn}

\begin{thm}[Kleitman~\cite{Kleitman}]
\label{Kequiv}
Let $P$ be a ranked poset. The following are equivalent: (1) $P$ has a regular covering by chains, (2) $P$ has the LYM property and (3) $P$ has the normalised matching property.
\end{thm}

The distribution $\mu$ on the maximal chains of $\{0,1,2\}^n$ that we construct essentially amounts to choosing a uniformly random chain from a regular covering by chains. One subtle point is that, in the proof of Theorem~\ref{multisetThm}, it will be useful to insist that the distribution $\mu$ is `memoryless' in the sense that
if $z_1<z_2<\dots<z_q$ in $\{0,1,2\}^n$ and $\mathcal{C}_\mu$ is a chain chosen randomly according to $\mu$, then
\begin{equation}\label{memoryless}\mathbb{P}\left(z_q\in\mathcal{C}_\mu\mid z_1,\dots,z_{q-1}\in\mathcal{C}_\mu\right)=\mathbb{P}\left(z_q\in\mathcal{C}_\mu\mid z_{q-1}\in\mathcal{C}_\mu\right).\end{equation}
Note that a chain chosen uniformly at random from a regular covering by chains need not satisfy \eqref{memoryless}. Also note that, if \eqref{memoryless} holds for every such sequence, then one can easily deduce
\[\mathbb{P}\left(z_1\in\mathcal{C}_\mu\mid z_2,\dots,z_{q}\in\mathcal{C}_\mu\right)=\mathbb{P}\left(z_1\in\mathcal{C}_\mu\mid z_{2}\in\mathcal{C}_\mu\right)\]
as well. Before constructing $\mu$, we require a couple of definitions.

\begin{defn}
For integers $i$ and $s$, let $L_{i}^s$ be the elements of $L_i$ with exactly $s$ coordinates equal to $2$. Note that $L_i^s$ is non-empty if and only if $\max\{0,i-n\}\leq s\leq \left\lfloor i/2\right\rfloor$. 
\end{defn} 

\begin{defn}
For integers $i$ and $s$, define 
\[L_{i}^{\geq s}:=\bigcup_{r=s}^{\left\lfloor i/2\right\rfloor} L_i^r\text{ and}\]
\[L_{i}^{\leq s}:=\bigcup_{r=\max\{0,i-n\}}^{s} L_i^r.\]
\end{defn}

Our goal is to construct a distribution $\mu$ on the maximal chains in $\{0,1,2\}^n$ such that \eqref{memoryless} holds and, for $0\leq i\leq 2n$, every element of $L_i$ is equally likely to be contained in $\mathcal{C}_\mu$. We begin by insisting that $\mathcal{C}_\mu$ contains $(0,\dots,0)$ with probability one and that it is equally likely to contain any element of $L_1$. 

Now, suppose that $\mathcal{C}_\mu$ has been constructed up to level $i$ where $1\leq i\leq 2n-2$. In order for \eqref{memoryless} to hold, we must choose an element of $L_{i+1}$ to add to $\mathcal{C}_\mu$ independently of the choice of $\mathcal{C}_\mu\cap\left(\bigcup_{j=1}^{i-1}L_j\right)$. That is, our goal is to define the values of $\mathbb{P}(y\in \mathcal{C}_\mu\mid x\in\mathcal{C}_\mu)$ for all $x\in L_i$ and $y\in L_{i+1}$ with $x<y$ in such a way that each element of $L_{i+1}$ is equally likely to be contained in $\mathcal{C}_\mu$.

We rephrase the problem of defining these conditional probabilities as follows. Suppose that we define a weighting of $L_{i+1}$ by assigning each element of $L_i$ a weight of $1/|L_i|$ and then redistributing the weight in such a way that each $x\in L_i$ sends weight $w_{x,y}\geq0$ to each $y\in L_{i+1}$ comparable to $x$, where $\sum_{\substack{y\in L_{i+1}\\ x<y}}w_{x,y}=1/|L_i|$. What we would like to do is to fix the weights $w_{x,y}$ in such a way that every element of $L_{i+1}$ receives weight $1/|L_{i+1}|$. Note that, for $\max\{0,i-n\}\leq s\leq \left\lfloor (i-2)/2\right\rfloor$ all of the weight sent by elements of $L_i^{\geq s+1}$ is received by elements of $L_{i+1}^{\geq s+1}$. So, the total weight received by the set $L_{i+1}^{\geq s+1}$ is 
\[\frac{\left|L_i^{\geq s+1}\right|}{|L_i|} + \sum_{\substack{(x,y)\in L_i^{s}\times L_{i+1}^{s+1}\\ x<y}}w_{x,y}.\]
In order to achieve the desired weighting of $L_{i+1}$, we need this quantity to be equal to $\frac{\left|L_{i+1}^{\geq s+1}\right|}{|L_{i+1}|}$, which implies that
\begin{equation}\label{more2s}\sum_{\substack{(x,y)\in L_i^{s}\times L_{i+1}^{s+1}\\ x<y}}w_{x,y} = \frac{\left|L_i\right|\left|L_{i+1}^{\geq s+1}\right|-\left|L_{i+1}\right|\left|L_i^{\geq s+1}\right|}{\left|L_i\right|\left|L_{i+1}\right|}.\end{equation}
Similarly since, for $\max\{0,i+1-n\}\leq s\leq \left\lfloor i/2\right\rfloor$, all of the weight sent by elements of $L_i^{\leq s-1}$ is received by elements of $L_{i+1}^{\leq s}$, we require that
\begin{equation}\label{same2s}\sum_{\substack{(x,y)\in L_i^{s}\times L_{i+1}^{s}\\ x<y}}w_{x,y} = \frac{|L_i|\left|L_{i+1}^{\leq s}\right|-\left|L_i^{\leq s-1}\right||L_{i+1}|}{|L_i||L_{i+1}|}.\end{equation}
We insist that our construction has the property that $w_{x,y}$ depends only on the number of coordinates of $x$ and $y$ which are equal to two. That is, if $x\in L_i^s$ and $y\in L_{i+1}^{s+1}$ for some $s$, then set $w_{x,y}=w_{i,s}$ and if $x\in L_i^s$ and $y\in L_{i+1}^s$, then set $w_{x,y}=w_{i,s}'$. The left sides of  \eqref{more2s} and \eqref{same2s} become $(i-2s)|L_i^s|w_{i,s}$ and $(n-i+s)|L_i^s|w_{i,s}'$, respectively. Rearranging, we get that, for  $\max\{0,i-n\}\leq s\leq \left\lfloor (i-1)/2\right\rfloor$, 
\begin{equation}\label{wExpression}w_{i,s}=\frac{\left|L_i\right|\left|L_{i+1}^{\geq s+1}\right| - \left|L_{i+1}\right|\left|L_i^{\geq s+1}\right|}{\left|L_i^s\right|\left|L_i\right|\left|L_{i+1}\right|(i-2s)}\end{equation}
and, for $\max\{0,i-n+1\}\leq s\leq \left\lfloor i/2\right\rfloor$, 
\begin{equation}\label{w'Expression}w_{i,s}'=\frac{\left|L_i\right|\left|L_{i+1}^{\leq s}\right| - \left|L_{i+1}\right|\left|L_i^{\leq s-1}\right|}{\left|L_i^s\right|\left|L_i\right|\left|L_{i+1}\right|(n-i+s)}.\end{equation}
For completeness, if $i$ is even, then set $w_{i,i/2}:=0$ and if $i\geq n$, set $w_{i,i-n}':=0$. The weights $w_{i,s}$ and $w_{i,s}'$ are all non-negative since $\{0,1,2\}^n$ has the normalised matching property.  

So, we extend the chain $\mathcal{C}_\mu$ to level $i+1$ in the following way. Given that $\mathcal{C}_\mu$ contains $x\in L_i^s$, then we add an element $y$ of $L_{i+1}$ comparable to $x$ to $\mathcal{C}_\mu$, independently of the choice of $\mathcal{C}_\mu\setminus\{x\}$, with probability $|L_i|w_{i,s}$ if $y\in L_{i+1}^{s+1}$ and with probability $|L_i|w_{i,s}'$ if $y\in L_{i+1}^s$. Note that $x$ is comparable to $i-2s$ elements of $L_{i+1}^{s+1}$ and $n-i+s$ elements of $L_{i+1}^s$ and that $(i-2s)w_{i,s} + (n-i+s)w_{i,s}'=1/|L_i| = \mathbb{P}(x\in\mathcal{C}_\mu)$ by definition. 

Let us verify that any element of $L_{i+1}$ is contained in $\mathcal{C}_\mu$ with probability $1/|L_{i+1}|$. By construction, given $y\in L_{i+1}^t$, the probability that $y$ is contained in $\mathcal{C}_\mu$ is
\[\sum_{\substack{x\in L_i \\ x<y}}\mathbb{P}\left(x,y\in\mathcal{C}_\mu\right) = \sum_{\substack{x\in L_i^{t-1} \\ x<y}}\mathbb{P}\left(x\in\mathcal{C}_\mu\right)\mathbb{P}\left(y\in\mathcal{C}\mid x\in\mathcal{C}_\mu\right)+\sum_{\substack{x\in L_i^{t} \\ x<y}}\mathbb{P}\left(x\in\mathcal{C}_\mu\right)\mathbb{P}\left(y\in\mathcal{C}\mid x\in\mathcal{C}_\mu\right).\]
If $t=0$, then the first sum is zero and the second is equal to $(i+1)w_{i,0}'=\frac{(i+1)\left|L_{i+1}^{\leq 0}\right|}{\left|L_i^0\right||L_{i+1}|(n-i)} = 1/|L_{i+1}|$. If $i$ is odd and $t=(i+1)/2$, then the second sum is zero and the first is equal to $\left(\frac{i+1}{2}\right)w_{i,(i-1)/2}= \frac{(i+1)\left|L_{i+1}^{\geq (i+1)/2}\right|}{2\left|L_i^{(i-1)/2}\right||L_{i+1}|} = 1/|L_{i+1}|$. Finally, if $\max\{i+1-n,1\} \leq t\leq \left\lfloor i/2\right\rfloor$, then the above expression becomes
\[tw_{i,t-1} + (i+1-2t)w_{i,t}'\]
\begin{align*}
&= \frac{t\left(\left|L_i\right|\left|L_{i+1}^{\geq t}\right| - \left|L_{i+1}\right|\left|L_i^{\geq t}\right|\right)}{\left|L_i^{t-1}\right|\left|L_i\right|\left|L_{i+1}\right|(i-2t+2)} + \frac{(i+1-2t)\left(\left|L_i\right|\left|L_{i+1}^{\leq t}\right| - \left|L_{i+1}\right|\left|L_i^{\leq t-1}\right| \right)}{\left|L_i^t\right|\left|L_i\right|\left|L_{i+1}\right|(n-i+t)}\\
&= \frac{\left(\left|L_i\right|\left|L_{i+1}^{\geq t}\right| - \left|L_{i+1}\right|\left|L_i^{\geq t}\right|\right)}{\binom{n}{i+1-t}\binom{i+1-t}{t}\left|L_i\right|\left|L_{i+1}\right|} + \frac{\left(\left|L_i\right|\left|L_{i+1}^{\leq t}\right| - \left|L_{i+1}\right|\left|L_i^{\leq t-1}\right| \right)}{\binom{n}{i+1-t}\binom{i+1-t}{t}\left|L_i\right|\left|L_{i+1}\right|}\\
&=\frac{|L_{i+1}^t|}{\binom{n}{i+1-t}\binom{i+1-t}{t}|L_{i+1}|}=\frac{1}{\left|L_{i+1}\right|}.\end{align*}
We complete the construction by insisting that $\mathcal{C}_\mu$ contains $(2,\dots,2)$ with probability one. 

Most of the work in this section is devoted to proving the following lemma. 

\begin{lem}
\label{conditionalProb}
Let $k\geq1$ be a fixed integer, let $n$ be large with respect to $k$ and let $\mathcal{C}_\mu$ be a chain chosen randomly according to $\mu$. Let $0\leq i,j\leq 2n$ be integers such that
\begin{itemize}
\item $|i-j|\geq k$, 
\item either $\min\{i,j\}\geq 2k$ or $\max\{i,j\}\leq 2n-2k$, and
\item $|j-n|\leq |i-n|$.
\end{itemize}
If $x\in L_i$ and $y\in L_j$ are such that $x$ is comparable to $y$, then 
\[\mathbb{P}\left(y\in \mathcal{C}_\mu\mid x\in\mathcal{C}_\mu\right)\leq \left(\frac{\ell_{3k-1}(n)}{\ell_{2k-1}(n)} - 1\right)^{-1}.\]
\end{lem}

We now derive Theorem~\ref{multisetThm} from Lemma~\ref{conditionalProb} after which we will turn our attention to proving the lemma itself.

\begin{proof}[Proof of Theorem~\ref{multisetThm} assuming Lemma~\ref{conditionalProb}]
For $0\leq r\leq k-1$, let $P_r$ be the poset consisting all elements of $\{0,1,2\}^n$ of rank equivalent to $n+\left\lceil\frac{2r-k+1}{2}\right\rceil$ modulo $k$ with partial order inherited from $\{0,1,2\}^n$. Define $S_r:=S\cap P_{r}$ and
\[t_r:=|S_r| - \ell_{n+\left\lceil\frac{2r-k+1}{2}\right\rceil}(n).\] 
Clearly, $\sum_{r=0}^{k-1}t_r = t$. We prove that, for each $r=0,\dots,k-1$,
\begin{equation}\label{theEllCaseMulti}\comp(S_r)\geq \left(\frac{\ell_{3k-1}(n)}{\ell_{2k-1}(n)} - 1\right)t_r.\end{equation} 

We prove \eqref{theEllCaseMulti} by induction on $t_r$ where the case $t_r\leq 0$ is trivial. By the inductive hypothesis, we can assume that every element of $S_r$ is comparable to fewer than $\frac{\ell_{3k-1}(n)}{\ell_{2k-1}(n)} - 1$ other elements of $S_r$. 

Next, we prove that it suffices to consider the case that $S_r$ contains no elements of rank less than $2k$. First, let $A$ be the set of all elements of $S_r$ of rank less than $k$. Clearly, each element of $A$ has rank $r$. Let us show that we can assume that $A$ is empty. For each subset $T$ of $A$, let $\Gamma(T)$ denote the elements of $L_{r+k}$ comparable to at least one member of $T$. If we had $|\Gamma(T)\setminus S_{r}|\geq |T|$ for every subset $T$ of $A$, then by Hall's Marriage Theorem~\cite{Hall} we could associate each element of $x\in A$ to an element $f(x)$ of $L_{r+k}\setminus S_{r}$ such that $x<f(x)$ and the function $f$ is injective. Given such a function $f$, it is clear that the set $S':=\left(S_r\setminus A\right)\cup f(A)$ has the same cardinality as $S_r$ and at most as many comparable pairs as $S_r$. 

Thus, we suppose that no such $f$ exists and so there there must be a set $T\subseteq A$ such that $|\Gamma(T)\setminus S_{r}| < |T|$. Since $\{0,1,2\}^n$ has the normalised matching property, we have
\[\frac{|\Gamma(T)\cap S_r|}{|T|} =\frac{|\Gamma(T)|}{|T|}- \frac{|\Gamma(T)\setminus S_r|}{|T|} > \frac{|L_{r+k}|}{|L_r|} -1\]
which is at least $\frac{\ell_{3k-1}(n)}{\ell_{2k-1}(n)}-1$ by Theorem~\ref{logConcave}. Putting these two facts together, we have
\begin{equation}\label{gammaT}|\Gamma(T)\cap S_r| > \left(\frac{\ell_{3k-1}(n)}{\ell_{2k-1}(n)}-1\right)|T|.\end{equation}
So, by the Pigeonhole Principle, there must be some element of $T$ which is comparable to more than $\frac{\ell_{3k-1}(n)}{\ell_{2k-1}(n)}-1$ elements of $S_r$, which is a contradiction. Thus, $S_r$ contains no elements of rank less than $k$. Applying the same argument to the elements of $S_r$ of rank $r+k$ (and using the fact that $S_r$ has no elements of rank $r$, which we just proved), we get that $S_r$ contains no elements of rank less than $2k$.

Now, we let $P_r^*$ be the poset obtained from $P_r$ by deleting all elements of rank less than $2k$. By the above argument, we can assume $S_r\subseteq P_r^*$. Let $D$ be a digraph with $V(D)=P_r^*$ containing every arc $(x,y)$ such that
\begin{itemize}
\item $x$ is comparable to $y$ and the rank of $y$ is closer to $n$ than the rank of $x$, or
\item $x<y$ and the rank of $x$ is $2n$ minus the rank of $y$. 
\end{itemize}
Let $\mathcal{C}_\mu$ be a random chain in $P_r^*$ obtained by taking a random chain in $\{0,1,2\}^n$ chosen according to $\mu$ and intersecting it with $P_r^*$. Provided that $n_0(k)$ is sufficiently large, the theorem now follows from Lemmas~\ref{randomCount} and~\ref{conditionalProb}.
\end{proof}

\subsection{Proof of Lemma~\ref{conditionalProb}}

It is clear by construction of $\mu$ that one can obtain an upper bound on  $\mathbb{P}\left(y\in \mathcal{C}_\mu\mid x\in \mathcal{C}_\mu\right)$ when $x$ and $y$ are in consecutive levels by bounding the `weight functions' $w_{i,s}$ and $w_{i,s}'$ from above. Our first step is to bound the weight functions relative to one another.

\begin{prop}
\label{wboundPROP}
For $1\leq i\leq n-1$ and $0\leq s < i/2$ we have $w_{i,s}' <  w_{i,s}$. 
\end{prop}

\begin{proof}
Define a weighting of the elements of $L_{i+1}$ by first assigning weight one to each element of $L_i$ and then redistributing the weight so that, for $0\leq s\leq \left\lfloor i/2\right\rfloor$, each $x\in L_i^s$ sends weight $|L_i|w_{i,s}'$ to each $y\in L_{i+1}^s$ comparable to $x$ and weight $|L_i|w_{i,s}$  to each $y\in L_{i+1}^{s+1}$ comparable to $x$. By construction of the weight functions $w_{i,s}$ and $w_{i,s}'$ we see that, after redistributing, each element of $L_i$ has weight zero and each element of $L_{i+1}$ has weight $|L_i|/|L_{i+1}|$.

Now, we define a second weighting of $L_{i+1}$ by assigning each element of $L_i$ weight one and redistributing the weight so that each element of $L_i^s$ sends its weight evenly to the elements of $L_{i+1}$ comparable to it. Under this weighting, each $y\in L_{i+1}^s$ receives weight
\[f(s):=\frac{s}{n+1-s} + \frac{i+1-2s}{n-s}.\]
We claim that $f(s)$ is strictly decreasing in $s$ for $s$ in the range $0\leq s\leq \frac{i+1}{2}$. Indeed, for $0\leq s\leq \frac{i-1}{2}$, we have
\begin{align*}
f(s+1)-f(s) &= \left(\frac{s+1}{n-s} + \frac{i-2s-1}{n-s-1}\right)-\left(\frac{s}{n-s+1} + \frac{i-2s+1}{n-s}\right)\\
&= \frac{s(n-i-2) - (n-i)(n+1)}{(n-s+1)(n-s)(n-s-1)}=\frac{(n-i)(s-n-1)-2s}{(n-s+1)(n-s)(n-s-1)}.\end{align*}
The denominator is always positive for $0\leq s\leq \frac{n-2}{2}$. If $i\in\{n-2,n-1\}$, then the numerator is clearly negative. Finally, if $1\leq i\leq n-3$, then the numerator is again negative since $0\leq s\leq\frac{i-1}{2} \leq \frac{n-2}{2}<n+1$ and $n-i-2<n-i$. Thus, $f(s)$ is strictly decreasing in $s$. 

Therefore, for $0\leq s\leq \frac{i-1}{2}$, the elements of $L_{i+1}^{\leq s}$ receive more weight on average under the second weighting than do the elements of $L_{i+1}^{\geq s+1}$.  Note that, under both redistribution rules, all of the weight from $L_i^{\leq s-1}$ is sent to $L_{i+1}^{\leq s}$ and all of the weight from $L_i^{\geq s+1}$ is sent to $L_{i+1}^{\geq s+1}$. Therefore, the amount of weight received by the elements of $L_{i+1}^{\geq s+1}$ versus the amount received by elements of $L_{i+1}^{\leq s}$ is only affected by the amount of weight that the elements of $L_i^s$ choose to send to $L_{i+1}^s$ and $L_{i+1}^{s+1}$. As we have mentioned, the first redistribution rule assigns each element of $L_{i+1}$ to the same weight while the second assigns more weight to elements of $L_{i+1}$ which have fewer coordinates equal to two. Putting this all together,  it must be the case that, under the first redistribution rule, each $x\in L_{i}^s$ sends more weight to each element of $L_{i+1}^{s+1}$ comparable to it than it sends to each element of $L_{i+1}^{s}$ comparable to it. This proves that $w_{i,s}' < w_{i,s}$. 
\end{proof}

We also use a similar `weight redistribution trick' to obtain various bounds on the ratio $|L_{i+1}|/|L_i|$ which will be needed later in this section. Since these proofs are all fairly routine, and not particularly enlightening, we have included them in the appendix.

In what follows, it will be useful to notice that the weight functions $w_{i,s}$ and $w_{i,s}'$ exhibit a certain symmetry around the $n$th level. Specifically,
\begin{gather*} w_{2n-i-1,n-i+s}' = \frac{\left|L_{2n-i-1}\right|\left|L_{2n-i}^{\leq n-i+s}\right|- \left|L_{2n-i}\right|\left|L_{2n-i-1}^{\leq n-i+s-1}\right|}{\left|L_{2n-i-1}^{n-i+s}\right||L_{2n-i-1}||L_{2n-i}|(s+1)}
\\ = \frac{\left|L_{i+1}\right|\left|L_{i}^{\leq s}\right|- \left|L_{i}\right|\left|L_{i+1}^{\leq s}\right|}{\left|L_{i}^{s}\right||L_{i+1}||L_{i}|(i-2s)}
 =w_{i,s}\numberthis\label{sym1}\end{gather*}
and, similarly,
\begin{equation}\label{sym2} w_{2n-i-1,n-i+s-1}=w_{i,s}'.\end{equation}
Next, we prove an upper bound on $w_{i,s}$ which is tight when $i$ is odd and $s=\frac{i-1}{2}$. An analogous result for the `upper half' of $\{0,1,2\}^n$ follows from \eqref{sym1}.

\begin{prop}
\label{wMaxBoth}
Suppose that $n\geq5$. Then for $1\leq i\leq n-1$ and $0\leq s < \frac{i}{2}$ we have $w_{i,s}\leq \frac{2}{(i+1)\left|L_{i+1}\right|}$.
\end{prop}

\begin{proof}
If there is an integer $s$ such that $0\leq s<\frac{i}{2}$ and
\begin{equation}
\begin{gathered}
\label{j0def}
2\left|L_i\right|s(s+1)(n-i+s) (n-i+s+1) \\
\leq(i-2s-1)(i-2s)(i-2s+1)\left((i+1)|L_{i+1}| - (n-i+s +1)|L_i|\right),
\end{gathered}
\end{equation}
then we let $s_0$ be the largest such integer. Otherwise, set $s_0:=0$. Note that both sides of \eqref{j0def} are non-negative for all $s$ such that $0\leq s< \frac{i}{2}$, where Lemma~\ref{ratio1} is used to verify that the right side is non-negative. If $s_0\geq1$, then the inequality \eqref{j0def} must hold for all $0\leq i\leq n-1$ and $s$ such that $0\leq s\leq s_0$ and must fail to hold for all $s$ such that $s_0+1\leq s< \frac{i}{2}$; this is because, for fixed $0\leq i\leq n-1$ and $s$ in the range $0\leq s<\frac{i}{2}$, the left side of \eqref{j0def} is increasing in $s$ and the right side is  decreasing in $s$. 

By \eqref{wExpression}, the inequality $w_{i,s}\leq \frac{2}{(i+1)\left|L_{i+1}\right|}$ can be rewritten as follows:
\[(i+1)\left(\left|L_i\right|\left|L_{i+1}^{\geq s+1}\right| - \left|L_{i+1}\right|\left|L_i^{\geq s+1}\right|\right) \leq 2\left|L_i^s\right|\left|L_i\right|(i-2s)\]
or, equivalently,
\begin{equation}\label{equiv}\left|L_i\right|\left((i+1)\left|L_{i+1}^{\geq s+1}\right| - 2(i-2s)\left|L_i^s\right|\right)\leq \left|L_{i+1}\right|(i+1)\left|L_i^{\geq s+1}\right|.\end{equation}
Our goal is to prove that \eqref{equiv} holds for $0\leq s< \frac{i}{2}$. The argument will either use induction on $s$ or  induction on $\left\lfloor\frac{i}{2}\right\rfloor - s$ depending on whether $s\leq s_0$ or $s\geq s_0+1$. Before proving the case $s\leq s_0$, though, we need to verify the `base case' $s=0$.

\begin{case2}
\label{j=0}
$s=0$.
\end{case2}
In this case, \eqref{equiv} becomes
\[\left|L_i\right|\left((i+1)\left|L_{i+1}\right| - (i+1)\left|L_{i+1}^0\right| - 2i\left|L_i^0\right|\right)\leq \left|L_{i+1}\right|(i+1)\left(\left|L_i\right| - \left|L_i^0\right|\right).\]
Since $(i+1)\left|L_{i+1}^0\right|= (i+1)\binom{n}{i+1} = (n-i)\binom{n}{i}=(n-i)\left|L_i^0\right|$, this simplifies to
\begin{equation}\label{notTooBig}(i+1)\left|L_{i+1}\right|\leq (n+i)\left|L_i\right|.\end{equation}
To prove \eqref{notTooBig}, we observe that each element of $L_i^r$ is comparable to $n-r$ elements of $L_{i+1}$ and each element of $L_{i+1}^r$ comparable to $i+1-r$ elements of $L_i$. Therefore, 
\[\sum_{r=0}^{\left\lfloor\frac{i+1}{2}\right\rfloor}(i+1-r)\left|L_{i+1}^r\right| = \sum_{r=0}^{\left\lfloor\frac{i}{2}\right\rfloor}(n-r)\left|L_{i}^r\right|.\]
This equality can be rewritten as follows:
\[(i+1)\left|L_{i+1}\right| - \sum_{r=0}^{\left\lfloor\frac{i+1}{2}\right\rfloor}r\left|L_{i+1}^r\right| = (n+i)\left|L_i\right| - \sum_{r=0}^{\left\lfloor\frac{i}{2}\right\rfloor}(i+r)\left|L_{i}^r\right|.\]
So, after a change of index, we see that \eqref{notTooBig} is equivalent to
\[\sum_{r=0}^{\left\lfloor\frac{i}{2}\right\rfloor}(i+r)\left|L_i^r\right| \geq \sum_{r=0}^{\left\lfloor\frac{i-1}{2}\right\rfloor}(r+1)\left|L_{i+1}^{r+1}\right|.\]
For $0\leq r\leq \left \lfloor \frac{i-1}{2}\right\rfloor$ we have
\begin{align*}
(i+r)\left|L_i^r\right| &= (i+r)\binom{n}{i-r}\binom{i-r}{r} = \frac{(i+r) \left|L_{i+1}^{r+1}\right|\binom{i-r}{r}}{\binom{i-r}{r+1}}\\
&=\frac{(i+r)(r+1)}{i-2r}\left|L_{i+1}^{r+1}\right|\geq (r+1)\left|L_{i+1}^{r+1}\right|\end{align*}
and so \eqref{notTooBig} holds. This completes the proof in this case.

\begin{case2}
$1\leq s\leq s_0$. 
\end{case2}

We can assume, by induction and the previous case, that \eqref{equiv} is true if $s$ is replaced by $s-1$. That is, we assume that the following inequality is true:
\[\left|L_i\right|\left((i+1)\left|L_{i+1}^{\geq s}\right| - 2(i-2s+2)\left|L_i^{s-1}\right|\right)\leq \left|L_{i+1}\right|(i+1)\left|L_i^{\geq s}\right|.\]
Subtracting this inequality from \eqref{equiv}, we see that it suffices to prove that
\[\left|L_i\right|\left((i+1)\left|L_{i+1}^{s+1}\right| - 2(i-2s)\left|L_i^s\right| + 2(i-2s+2)\left|L_i^{s-1}\right|\right) \]
\[\leq \left|L_{i+1}\right|(i+1)\left|L_i^{s+1}\right|.\]
After substituting the values of $\left|L_{i+1}^{s+1}\right|$, $\left|L_i^s\right|$, $\left|L_i^{s-1}\right|$ and $\left|L_i^{s+1}\right|$ into the above inequality, it becomes
\[\left|L_i\right|\left((i+1)\binom{n}{i-s}\binom{i-s}{s+1} - 2(i-2s)\binom{n}{i-s}\binom{i-s}{s} + 2(i-2s+2)\binom{n}{i-s+1}\binom{i-s+1}{s-1}\right)\]
\[\leq \left|L_{i+1}\right|(i+1)\binom{n}{i-s-1}\binom{i-s-1}{s+1}.\]
By simplifying this expression, one can show that it is in fact equivalent to \eqref{j0def}, which holds for $s$ because $s\leq s_0$. This completes the proof in this case. 

\begin{case2}
$s_0+1\leq s< \frac{i}{2}$. 
\end{case2}

In the case that $i$ is odd and $s=\frac{i-1}{2}$ we have $w_{i,s}=\frac{2}{(i+1)\left|L_{i+1}\right|}$ by definition and so the result holds with equality. If $i$ is even and $s=\frac{i-2}{2}$, then $w_{i,s} = \frac{2}{i|L_{i+1}|} - \frac{4}{(2n-i)|L_{i}|}$ by \eqref{wExpression}. In this case, the bound $w_{i,s}\leq \frac{2}{(i+1)\left|L_{i+1}\right|}$ reduces to the following:
\[(2n-i)|L_i|\leq 2(i+1)|L_{i+1}|.\]
Since $1\leq i\leq n-1$, this inequality holds by Lemmas~\ref{ratio1},~\ref{ration-2} and~\ref{ration-1}. 

So, we are done unless $s_0< s < \frac{i}{2}-1$. We can assume, by induction on $\left\lfloor\frac{i}{2}\right\rfloor-s$, that \eqref{equiv} is true if $s$ is replaced by $s+1$. That is, we assume that the following inequality holds:
\[\left|L_i\right|\left((i+1)\left|L_{i+1}^{\geq s+2}\right| - 2(i-2s-2)\left|L_i^{s+1}\right|\right)\leq \left|L_{i+1}\right|(i+1)\left|L_i^{\geq s+2}\right|.\]
Similar to the proof of the previous case, we can derive \eqref{equiv} from the above inequality provided that we can prove the following: 
\[\left|L_i\right|\left((i+1)\left|L_{i+1}^{s+2}\right| - 2(i-2s-2)\left|L_i^{s+1}\right| + 2(i-2s)\left|L_i^s\right|\right)\]
\[> \left|L_{i+1}\right|(i+1)\left|L_i^{s+2}\right|.\]
Simplifying this expression in a similar fashion to the proof of the previous case, we see that it is equivalent to the following inequality:
\[2\left|L_i\right|(s+1)(s+2)(n-i+s+1)(n-i+s+2)\]
\[> (i-2s-3)(i-2s-2)(i-2s-1)\left((i+1)\left|L_{i+1}\right| - (n-i+s+2)\left|L_i\right|\right).\]
This is precisely the negation of \eqref{j0def} with $s$ is replaced by $s+1$. This inequality holds because $s>s_0$ and $s_0$ was chosen to be the largest index satisfying \eqref{j0def}. This completes the proof of the proposition.
\end{proof}

We derive the following corollary.

\begin{cor}
\label{OneStep}
For $n\geq5$ and $0\leq i\leq 2n-1$, if $x\in L_i$ and $y\in L_{i+1}$ such that $x<y$, then
\[\mathbb{P}\left(y\in \mathcal{C}_\mu\mid x\in\mathcal{C}_\mu\right) \leq 
\begin{cases}	2|L_i|/(i+1)|L_{i+1}| 	& \text{if }i\leq n-1 \\
							2/(2n-i)					& \text{otherwise.}\end{cases}\]
\end{cor}

\begin{proof}
Let $\max\{0,i-n\}\leq s\leq i/2$ so that $x\in L_i^s$. By construction, $\mathbb{P}\left(y\in \mathcal{C}_\mu\mid x\in\mathcal{C}_\mu\right)$ is either equal to $w_{i,s}|L_i|$ or $w_{i,s}'|L_i|$. 

Suppose first that $i\leq n-1$. If $s<i/2$, then by Propositions~\ref{wboundPROP} and~\ref{wMaxBoth}, we have that $w_{i,s}$ and $w_{i,s}'$ are both bounded above by $\frac{2}{(i+1)|L_{i+1}|}$ and we are done. In the case that $i$ is even and $s=i/2$, we have that $w_{i,i/2}=0$ and $w_{i,i/2}' = \frac{2}{2n-i}$ which is bounded above by $\frac{2}{(i+1)|L_{i+1}|}$ by Lemmas~\ref{ratio1},~\ref{ration-2} and~\ref{ration-1}. This completes the proof for $i\leq n-1$. The case $i\geq n$ follows from the case $i\leq n-1$ via equations \eqref{sym1} and \eqref{sym2} and the fact that $\mathbb{P}\left(y\in \mathcal{C}_\mu\mid x\in\mathcal{C}_\mu\right) = \mathbb{P}\left(x\in \mathcal{C}_\mu\mid y\in\mathcal{C}_\mu\right)|L_{i+1}|/|L_i|$. 
\end{proof}

We are now ready to prove Lemma~\ref{conditionalProb}. The proof is divided into three cases depending on whether
\begin{itemize}
\item $|j-i|$ is `small,' 
\item $|j-i|$ is `large' and $i+j=2n$, or 
\item $|j-i|$ is `large' and $i+j\neq 2n$.
\end{itemize}
The first case is done by summing over all chains containing $x$ and $y$ and applying \eqref{memoryless} and Corollary~\ref{OneStep} to each of them. In the second case, we assume that $(x,y)$ is chosen from $L_i\times L_j$ to maximise $\mathbb{P}\left(y\in\mathcal{C}_\mu\mid x\in\mathcal{C}_\mu\right)$ and obtain a lower bound on the number of $y'\in L_j$ comparable to $x$ such that $\mathbb{P}\left(y'\in\mathcal{C}_\mu\mid x\in\mathcal{C}_\mu\right)=\mathbb{P}\left(y\in\mathcal{C}_\mu\mid x\in\mathcal{C}_\mu\right)$. Since the sum of $\mathbb{P}\left(y'\in\mathcal{C}_\mu\mid x\in\mathcal{C}_\mu\right)$ over all such $y'$ is at most one, this lower bound on the number of such $y'$ gives us an upper bound on $\mathbb{P}\left(y\in\mathcal{C}_\mu\mid x\in\mathcal{C}_\mu\right)$. The third case is done by a simple inductive trick which uses the first two cases.

\begin{proof}[Proof of Lemma~\ref{conditionalProb}]
Let $k\geq 1$ be a fixed integer and let $n$ be sufficiently large with respect to $k$. Let $i,j$ be integers satisfying the hypotheses of the lemma. Note that, by the symmetry of $\mu$ around level $n$ (i.e. by equations \eqref{sym1} and \eqref{sym2}), we can assume that $i<n$. The assumption $|j-n|\leq |i-n|$ can therefore be rewritten as $i<j$ and $i+j\leq 2n$. 

\begin{case4}
\label{smallGap}
$k\leq j-i\leq 11n/10$.
\end{case4}

It is clear that $\mathbb{P}\left(y\in \mathcal{C}_\mu\mid x\in\mathcal{C}_\mu\right)$ is equal to the sum of $\mathbb{P}\left(z_{i+1},\dots,z_{j}\in \mathcal{C}_\mu\mid z_i\in\mathcal{C}_\mu\right)$ over all sequences $z_i<z_{i+1}<\dots<z_{j}$ such that $z_i=x$ and $z_{j}=y$. The number of such sequences is clearly bounded above by $(j-i)!$ and, for every such sequence, we have
\[\mathbb{P}\left(z_{i+1},\dots,z_{j}\in \mathcal{C}_\mu\mid z_i\in\mathcal{C}_\mu\right) = \prod_{r=i}^{j-1}\mathbb{P}\left(z_{r+1}\in \mathcal{C}_\mu\mid z_r\in\mathcal{C}_\mu\right)\]
by \eqref{memoryless}. For $i\leq r\leq n-1$, we have that $\mathbb{P}\left(z_{r+1}\in \mathcal{C}_\mu\mid z_r\in\mathcal{C}_\mu\right) \leq \frac{2|L_i|}{(i+1)|L_{i+1}|}\leq \frac{2}{n}$ by Corollary~\ref{OneStep} and Lemmas~\ref{ratio1},~\ref{ration-2} and~\ref{ration-1} and for $n\leq r\leq j-1$ we have $\mathbb{P}\left(z_{r+1}\in \mathcal{C}_\mu\mid z_r\in\mathcal{C}_\mu\right) \leq \frac{2}{2n-r}$ by Corollary~\ref{OneStep}. Since $i+j\leq 2n$, we have that $i\leq n-\frac{j-i}{2}$ and so we get 
\begin{align*}
\mathbb{P}\left(z_{i+1},\dots,z_{j}\in \mathcal{C}_\mu\mid z_i\in\mathcal{C}_\mu\right)& \leq \left(\frac{2}{n}\right)^{\left\lceil (j-i)/2\right\rceil}\frac{2^{\left\lfloor (j-i)/2\right\rfloor}}{n(n-1)\cdots (n-\left\lfloor (j-i)/2\right\rfloor + 1)}\\
& = \frac{2^{j-i} (n-\left\lfloor(j-i)/2\right\rfloor)!}{n^{\left\lceil (j-i)/2\right\rceil}n!}.\end{align*}
As mentioned earlier, the number of such sequences is at most $(j-i)!$ and so we obtain
\begin{equation}\label{j-iBound}\mathbb{P}\left(y\in \mathcal{C}_\mu\mid x\in\mathcal{C}_\mu\right) \leq \frac{2^{j-i} (j-i)!(n-\left\lfloor(j-i)/2\right\rfloor)!}{n^{\left\lceil (j-i)/2\right\rceil}n!}.\end{equation}
In the case $j-i=k$, the right side is at most $2^{k}\binom{n}{k}^{-1}$ which, by \eqref{ratioAsymp} and the fact that $\frac{\ell_2(n)}{\ell_1(n)} - 1 = \frac{n-1}{2}$, is less than $\left(\frac{\ell_{3k-1}(n)}{\ell_{2k-1}(n)} - 1\right)^{-1}$ for $n_0(k)$ sufficiently large and so we are done. Also, the right side is decreasing in $j-i$ for $k\leq j-i\leq n/10$ and so we are done for $j-i$ in this range as well. 

So, we assume that $n/10 < j-i \leq 11n/10$. If we define $c:=\frac{j-i}{n}$, then \eqref{j-iBound} can be rewritten as
\[\mathbb{P}\left(y\in \mathcal{C}_\mu\mid x\in\mathcal{C}_\mu\right) \leq \frac{2^{cn} (cn)!(n-\left\lfloor cn/2\right\rfloor)!}{n^{\left\lceil cn/2\right\rceil}n!}.\] 
Applying Stirling's Approximation, we see that the right side is bounded above by
\[O\left(\frac{\sqrt{n} 2^{cn}\left(\frac{cn}{e}\right)^{cn}\left(\frac{n(1-c/2)}{e}\right)^{n(1-c/2)}}{n^{cn/2} \left(\frac{n}{e}\right)^n}\right)= O\left(\sqrt{n}\left(\frac{\left(2c\right)^{c}\left(1-c/2\right)}{e^{c/2}(1-c/2)^{c/2}}\right)^n\right).\]
Notice that $\frac{\left(2c\right)^{c}\left(1-c/2\right)}{e^{c/2}(1-c/2)^{c/2}}<0.99$ for $1/10 < c \leq 11/10$ and so the above expression decreases exponentially with $n$ and therefore is less than $\left(\frac{\ell_{3k-1}(n)}{\ell_{2k-1}(n)} - 1\right)^{-1}$ for $n_0(k)$ sufficiently large. This completes the proof in this case.

\begin{case4}
\label{bigSymmetricGap}
$i+j=2n$ and $i < 9n/20$.
\end{case4}

In this case, we suppose further that the pair $(x,y)$ is chosen from $L_i\times L_{2n-i}$ so that $\mathbb{P}\left(y\in\mathcal{C}_\mu\mid x\in\mathcal{C}_\mu\right)$ is maximised. We prove the following claim.

\begin{claim}
\label{1To2}
If there exist a coordinate $\alpha\in[n]$ such that $x_\alpha=0$ and $y_\alpha=2$, then for every coordinate $\beta\in[n]$ with $x_\beta=1$ we have $y_\beta=2$. 
\end{claim}

\begin{proof}
Suppose to the contrary that there is a coordinate $\alpha\in[n]$ such that $x_\alpha=0$ and $y_\alpha=2$ and a coordinate $\beta\in[n]$ with $x_\beta=y_\beta=1$. Let $\tau_{\alpha,\beta}:\{0,1,2\}^n\to\{0,1,2\}^n$ be the function which exchanges the $\alpha$th coordinate with the $\beta$th coordinate. Note that $x<\tau_{\alpha,\beta}(y)$. Our goal is to show that 
\[\mathbb{P}\left(\tau_{\alpha,\beta}(y)\in\mathcal{C}_\mu\mid x\in\mathcal{C}_\mu\right)> \mathbb{P}\left(y\in\mathcal{C}_\mu\mid x\in\mathcal{C}_\mu\right)\]
which will contradict our choice of $x$ and $y$.

Given a sequence $Z=(z_i,\dots,z_{2n-i})$ of elements of $\{0,1,2\}^n$ with $z_i<\dots<z_{2n-i}$, $z_i=x$ and $z_{2n-i}=y$ define $\ell(Z)$ to be the smallest integer $\ell$ such that the $\alpha$th coordinate of $z_\ell$ is a two. For every such sequence, define another sequence $\pi(Z)$ as follows:
\[\pi(Z):=\left(z_i,\dots,z_{\ell(Z)-1},\tau_{\alpha,\beta}\left(z_{\ell(Z)}\right),\dots,\tau_{\alpha,\beta}\left(z_{2n-i}\right)\right).\]
Notice that $\pi(Z)$ is a chain starting with $x$ and ending at $\pi_{\alpha,\beta}(y)$. It is not hard to see that the function $\pi$ is injective. Now, the key observation is that, by construction of $\mathcal{C}_\mu$, for any such sequence $Z$, we have
\[\mathbb{P}\left(z_{i+1},\dots,z_{2n-i}\in \mathcal{C}_\mu\mid z_i\in\mathcal{C}_\mu\right)\]
\[ =\mathbb{P}\left(z_{i+1},\dots,z_{\ell(Z)-1},\tau_{\alpha,\beta}\left(z_{\ell(Z)}\right),\dots,\tau_{\alpha,\beta}\left(z_{2n-i}\right)\in \mathcal{C}_\mu\mid z_i\in\mathcal{C}_\mu\right).\]
This proves $\mathbb{P}\left(\tau_{\alpha,\beta}(y)\in\mathcal{C}_\mu\mid x\in\mathcal{C}_\mu\right)\geq \mathbb{P}\left(y\in\mathcal{C}_\mu\mid x\in\mathcal{C}_\mu\right)$ since $\mathbb{P}\left(y\in \mathcal{C}_\mu\mid x\in\mathcal{C}_\mu\right)$ is equal to the sum over all such sequences of $\mathbb{P}\left(z_{i+1},\dots,z_{2n-i}\in \mathcal{C}_\mu\mid z_i\in\mathcal{C}_\mu\right)$. In order to prove the strict inequality, all that we need is a single chain containing $x$ and $\tau_{\alpha,\beta}(y)$ which is not a subsequence of one of the chains $\pi(Z)$ for $Z$ as above. Any chain of the form $x<z<\tau_{\alpha,\beta}(y)$ with $z_\beta=2$ has this property. This completes the proof of the claim.
\end{proof}

We now prove that, in fact, there must be at least $n/10$ coordinates $\alpha\in[n]$ such that $x_\alpha=0$ and $y_\alpha=2$. If not, then we must have that the rank of $y$ is at most $2i + (n-i) + n/10 = 11n/10+i$. However, by assumption of this case, the rank of $y$ is precisely $2n-i$. Putting this together, we get that $2n-i\leq 11n/10 +i$ which implies that $i\geq 9n/20$, a contradiction. Thus, by Claim~\ref{1To2} we have that $y_\beta=2$ for every coordinate $\beta\in[n]$ with $x_\beta\geq 1$. 

By the result of the previous paragraph and the symmetry of the distribution $\mu$, we see that $\mathbb{P}\left(y'\in \mathcal{C}_\mu\mid x\in \mathcal{C}_\mu\right)$ is equal to $\mathbb{P}\left(y\in \mathcal{C}_\mu\mid x\in \mathcal{C}_\mu\right)$ for every $y'\in L_{2n-i}$ such that
\begin{itemize}
\item $y'_\beta = 2$ for every  coordinate $\beta\in[n]$ with $x_\beta\geq1$, and
\item $y'$ has the same number of coordinates in $\{0,1\}$ as $y$. 
\end{itemize}
By the hypotheses of the lemma and the fact that the rank of $x$ plus the rank of $y$ is $2n$, we must have that the rank of $y$ is at most $2n-2k$. This implies that $y$ must have at least $k$ coordinates in $\{0,1\}$. Also, recall that we have proved above that there are at least $n/10\gg k$ coordinates such that $\alpha\in[n]$ such that $x_\alpha=0$ and $y_\alpha=2$. Since $i<9n/20$, we have that $x$ has at least $11n/20$ coordinates which are equal to zero. Putting all of this together, we get that, for $n_0(k)$ sufficiently large, the number of such $y'$ is at least
\[\binom{\left\lceil 11n/20\right\rceil}{k} > \frac{\ell_{3k-1}(n)}{\ell_{2k-1}(n)} - 1\]
by \eqref{ratioAsymp}. Since $\mathbb{P}\left(y'\in \mathcal{C}_\mu\mid x\in \mathcal{C}_\mu\right)$ is the same for all such $y'$ and the sum of this quantity over all $y'$ is at most one, we must have that $\mathbb{P}\left(y\in \mathcal{C}_\mu\mid x\in \mathcal{C}_\mu\right)\leq \left(\frac{\ell_{3k-1}(n)}{\ell_{2k-1}(n)} - 1\right)^{-1}$. This completes the proof in this case.

\begin{case4}
$j-i> 11n/10$ and $i+j<2n$.
\end{case4}

In this case, we proceed by induction on $2n-i-j$. The basis of our induction will be the case $2n-i-j=0$, which is covered by Case~\ref{bigSymmetricGap}. We can write $\mathbb{P}\left(y\in \mathcal{C}_\mu\mid x\in\mathcal{C}_\mu\right)$ as
\[\sum_{\substack{z\in L_{i+1} \\ x<z<y}}\mathbb{P}\left(y,z\in \mathcal{C}_\mu\mid x\in\mathcal{C}_\mu\right)\]
\[=\sum_{\substack{z\in L_{i+1} \\ x<z<y}}\mathbb{P}\left(z\in \mathcal{C}_\mu\mid x\in\mathcal{C}_\mu\right)\mathbb{P}\left(y\in \mathcal{C}_\mu\mid z\in\mathcal{C}_\mu\right)\]
by \eqref{memoryless}. Since  $\sum_{\substack{z\in L_{i+1} \\ x<z<y}}\mathbb{P}\left(z\in \mathcal{C}_\mu\mid x\in\mathcal{C}_\mu\right)\leq 1$, we see that it suffices to show that 
\begin{equation}\label{zBound}\mathbb{P}\left(y\in \mathcal{C}_\mu\mid z\in\mathcal{C}_\mu\right)\leq\left(\frac{\ell_{3k-1}(n)}{\ell_{2k-1}(n)} - 1\right)^{-1}\end{equation}
for every $z\in L_{i+1}$ such that $x<z<y$. So, let $z$ be any such element. By assumption, we have that $i+1<j$ and $(i+1)+j\leq 2n$ and so the rank of $y$ is at least as close to $n$ as the rank of $z$ is. If we have $j-(i+1) \leq 11n/10$, then \eqref{zBound} holds by the argument given in Case~\ref{smallGap} and we are done. So, we assume that $j-(i+1) > 11n/10$. If we have $(i+1)+j=2n$, then \eqref{zBound} holds by the argument given in Case~\ref{bigSymmetricGap}. Thus, we have that $j-(i+1)>11n/10$ and $(i+1)+j<2n$ and we now get that \eqref{zBound} holds by the inductive hypothesis. This completes the proof of the lemma.
\end{proof}

\section{Containers for Antichains}
\label{containerSec}

In this section, we generalise a container-type lemma for the boolean lattice proved by Balogh, Mycroft and Treglown~\cite{RandomSperner} to general posets. The main idea in the proof originated in the work of Kleitman and Winston~\cite{KleitmanWinston} and has now been used many times; see, e.g.,~\cite{RandomSperner,BaloghWagnerBoolean,GraphCont}. In essence, what this lemma says is that if $P$ satisfies a certain supersaturation bound, then there is a collection of `not very large' subsets of $P$ indexed by `very small' subsets of $P$ with the property that every antichain in $P$ is contained in at least one set from the collection. The elements of such a collection are referred to as \emph{containers}. The fact that the containers are indexed by small subsets of $P$ is typically used to argue that the total number of containers is very small compared to the total number of antichains.

\begin{defn}
Given a set $X$ and an integer $j$, we let $\binom{X}{\leq j}$ denote the set of all subsets of $m$ of cardinality at most $j$. Also, for $n\geq j$, define $\binom{n}{\leq j}:=\sum_{r=0}^j\binom{n}{r}$. 
\end{defn}

\begin{lem}
\label{generalCont}
Let $d$ and $m$ be positive integers and let $P$ be a poset such that $|P|> m$ and every subset $S$ of $P$ of cardinality greater than $m$ contains at least $|S|d$ comparable pairs. Then there exists a function
\[f:\binom{P}{\leq |P|/(2d+1)}\to\binom{P}{\leq m}\]
such that for every antichain $I\subseteq P$ there exists a subset $T$ of $I$ of cardinality at most $|P|/(2d+1)$ with $T\cap f(T)=\emptyset$ and $I\subseteq T\cup f(T)$. 
\end{lem}

A key idea from~\cite{RandomSperner} is that it can often be advantageous to prove a `multi-stage' container lemma, where the purpose of the early stages is to `thin out' the poset to reduce the overall number of containers (see also~\cite{BaloghWagnerBoolean}). We prove the following lemma of this type, which implies Lemma~\ref{generalCont}.

\begin{lem}
\label{generalContK}
For $k\geq1$ let $d_1> \dots> d_k$ and $m_0> \dots> m_k$ be positive integers and let $P$ be a poset such that $|P|=m_0$ and, for $1\leq j\leq k$, every subset $S$ of $P$ of cardinality greater than $m_j$ contains at least $|S|d_j$ comparable pairs. Then there exists functions $f_1,\dots,f_k$ where
\[f_j:\binom{P}{\leq \sum_{r=1}^j \left(m_{r-1}/(2d_r + 1)\right)}\to\binom{P}{\leq m_j}\]
such that for every antichain $I\subseteq P$ there exists disjoint subsets $T_1,\dots,T_k$ of $I$ with
\begin{enumerate}[(i)]
\item \label{CONTeachSSmall}$|T_{j}|\leq m_{j-1}/(2d_{j}+1)$ for $1\leq j\leq k$,
\item \label{CONTSinf}$T_{j}\subseteq f_{j-1}\left(\bigcup_{r=1}^{j-1}T_r\right)$ for $2\leq j\leq k-1$, 
\item \label{CONTdisjoint}$\left(\bigcup_{r=1}^j T_r\right)\cap f_j\left(\bigcup_{r=1}^j T_r\right)=\emptyset$ for $1\leq j\leq k$ and
\item\label{CONTcontainersContain} $I\subseteq \left(\bigcup_{r=1}^kT_r\right) \cup f_k\left(\bigcup_{r=1}^kT_r\right)$.
\end{enumerate}
\end{lem}

\begin{proof}
Fix an arbitrary total order $x_1,\dots,x_{|P|}$ on the elements of $P$ (this will only be used to `break ties' later in the proof). The proof of this lemma amounts to a simple application of the so-called Kleitman--Winston algorithm~\cite{KleitmanWinston}. This algorithm takes, as an input, an antichain $I$ of $P$ and produces the sets $T_1,\dots,T_k$ as well as the sets $f_1(T_1), \dots,f_k\left(\bigcup_{r=1}^kT_r\right)$. After describing this algorithm,  we will verify that these sets have all of the desired properties and that the functions $f_1,\dots,f_k$ are well defined; i.e., we show that $f_j\left(\bigcup_{r=1}^jT_r\right)$ depends only on $\bigcup_{r=1}^jT_r$ and not on the input antichain $I$ or on the sets $T_1,\dots,T_j$ individually.

Given a set $S\subseteq P$ and $x\in S$, define $N_{S}(x)$ to be the set of elements of $S$ that are comparable to $x$ and let $d_{S}(x):=|N_{S}(x)|$. Set $P_0:=P$ and let $u_0$ be the element of $P_0$ such that $d_{P_0}(u_0)$ is maximum; if there is a tie, then we let $u_{0}$ be the choice which comes earliest in the total order. Now, given $i\geq1$, the set $P_{i-1}$ and an element $u_{i-1}$ of $P_{i-1}$, the $i$th step of the algorithm will either terminate or it will produce a non-empty set $P_i\subsetneq P_{i-1}$ and an element $u_i$ of $P_i$ which it will pass to the next step. We proceed differently depending on the value of $d_{P_{i-1}}\left(u_{i-1}\right)$.

\begin{case}
\label{bigdegcase}
$d_{P_{i-1}}(u_{i-1})\geq 2d_k$.
\end{case}

Let $j$ be the smallest integer in $\{1,\dots,k\}$ such that $d_{P_{i-1}}(u_{i-1})\geq 2d_j$. If $u_{i-1}\notin I$, then we define $P_i:=P_{i-1}\setminus\{u_{i-1}\}$. On the other hand, if $u_{i-1}\in I$, then we add $u_{i-1}$ to $T_j$ and define $P_i:=P_{i-1}\setminus\left(\{u_{i-1}\}\cup N_{P_{i-1}}(u_{i-1})\right)$. It is clear that $P_i\subsetneq P_{i-1}$.

If $P_i\neq\emptyset$, then let $u_i\in P_i$ such that $d_{P_i}(u_i)$ is maximum and, subject to this, $u_i$ comes earliest in the total order. If, in addition, $d_{P_i}(u_i)\geq 2d_j$, then we simply proceed to the next step of the algorithm. 

Suppose now that $P_i=\emptyset$ or that $d_{P_i}(u_i)< 2d_j$. In the former case, set $\ell:=k-j$ and, in the latter case, let $\ell\in\{0,\dots,k-j\}$ be the largest integer such that $d_{P_i}(u_i)<2d_{j+\ell}$. We do the following:
\begin{itemize}
\item terminate the definition of $T_j$, 
\item set $T_{j+1},\dots,T_{j+\ell}:=\emptyset$,
\item set $f_{j}\left(\bigcup_{r=1}^{j}T_r\right) =\dots =  f_{j+\ell}\left(\bigcup_{r=1}^{j+\ell}T_r\right):=P_{i}$.
\end{itemize}
If $P_i=\emptyset$, then terminate the algorithm; otherwise, proceed to the next step.

\begin{case}
\label{smalldegcase}
$d_{P_{i-1}}\left(u_{i-1}\right) < 2d_k$.
\end{case}

In this case, we simply terminate the algorithm. This concludes the description of the algorithm. 

A key feature of the algorithm is that the sequence $d_{P_0}(u_0),d_{P_1}(u_1),\dots$ is non-increasing, which follows easily from the choice of the elements $u_0,u_1,\dots$ and the fact that $P_{i}\subsetneq P_{i-1}$ for all $i\geq1$. Since this sequence is non-increasing and $d_1>\dots>d_k$ we see that, for each $j\in\{1,\dots,k\}$, there is a unique step of the algorithm such that the definition of $T_j$ and $f_j\left(\bigcup_{r=1}^j T_r\right)$ is finalised. Also, if $1\leq j < j' \leq k$, then $T_j$ and $f_j\left(\bigcup_{r=1}^j T_r\right)$ are defined before $T_{j'}$ and $f_{j'}\left(\bigcup_{r=1}^{j'}T_r\right)$.

It is clear by construction that \eqref{CONTSinf} and \eqref{CONTdisjoint} hold and that $T_j\subseteq I$ for $1\leq j\leq k$. Also, \eqref{CONTcontainersContain} holds since we never delete an element of $I$ during the running of the algorithm. By construction, every element of $f_j\left(\bigcup_{r=1}^jT_r\right)$ is comparable to fewer than $2d_j$ other elements in $f_j\left(\bigcup_{r=1}^jT_r\right)$. Therefore, $\comp\left(f_j\left(\bigcup_{r=1}^jT_r\right)\right)< \left|f_j\left(\bigcup_{r=1}^jT_r\right)\right|d_j$ which implies that $\left|f_j\left(\bigcup_{r=1}^jT_r\right)\right|\leq m_j$ by hypothesis. This verifies that $f_j$ does, indeed, map into $\binom{P}{\leq m_j}$. For each element that we added to $T_1$, we deleted at least $2d_1+1$ elements of $P$. Therefore, $|T_1|\leq |P|/(2d_1+1) = m_{0}/(2d_1+1)$. Similarly, for $j\geq2$, for each element that we added to $T_j$ we deleted at least $2d_j+1$ elements of $f_{j-1}\left(\bigcup_{r=1}^{j-1}T_r\right)$ and so  $|T_j|\leq m_{j-1}/(2d_j+1)$. Thus, \eqref{CONTeachSSmall} holds. 

Finally, we argue that the functions $f_1,\dots,f_k$ are well defined. Let $I$ and $I'$ be antichains yielding the same set $\bigcup_{r=1}^jT_r$ for some $j$ with $1\leq j\leq k$. Let $u_0,u_1,\dots$ and $P_0,P_1,\dots$ be the elements and subsets of $P$ produced while running the algorithm with input $I$ and let $u_0',u_1',\dots$ and $P_0',P_1',\dots$ be the elements and subsets of $P$ produced while running the algorithm with input $I'$. Let $i$ denote the minimum integer such that $d_{P_i}(u_i) < 2d_j$. We prove that $u_t=u_t'$ and $P_t=P_t'$ for all $0\leq t\leq i$. If not, let $t$ be the smallest such integer for which it fails. Clearly, $t\geq1$. Now, if $P_t=P_t'$, then $u_t= u_t'$ follows immediately so we must have $P_t\neq P_t'$. However, since $u_{t-1}=u_{t-1}'$ and $P_{t-1}=P_{t-1}'$, the only way that we can have $P_t\neq P_t'$ is if $u_{t-1}\in I$ and $u_{t-1}'\notin I'$ or vice versa. By definition of $i$,  we have $d_{P_{t-1}}(u_{t-1})\geq 2d_j$ and so  $u_{t-1}\in I$ implies $u_{t-1}\in \bigcup_{r=1}^jT_r$ by the description of the algorithm. However, $u_{t-1}'\notin I'$ implies that the algorithm will not add $u_{t-1}'$ to any of the sets $T_1,\dots,T_k$ and so $u_{t-1}'\notin \bigcup_{r=1}^jT_r$, contradicting the fact that $u_{t-1}=u_{t-1}'$. Therefore $f_j\left(\bigcup_{r=1}^jT_r\right)$ does not depend on $I$. This argument also proves that it depends only on the union $\bigcup_{r=1}^jT_r$ and not on the sets $T_1,\dots,T_j$ individually. This completes the proof.
\end{proof}

The following lemma is a consequence of Lemma~\ref{generalContK} which we will use to count antichains in the next section. 

\begin{lem}
\label{containersCor}
For $k\geq1$ let $d_1> \dots> d_k$ and $m_0> \dots> m_k$ be positive integers and let $P$ be a poset such that $|P|=m_0$ and, for $1\leq j\leq k$, every subset $S$ of $P$ of cardinality greater than $m_j$ contains at least $|S|d_j$ comparable pairs. Then there is a collection $\mathcal{F}$ of subsets of $P$ such that
\begin{enumerate}[(a)]
\item \label{collectionBound}$|\mathcal{F}|\leq \prod_{r=1}^k\binom{m_{r-1}}{\leq m_{r-1}/(2d_r+1)}$,
\item \label{setBound}$|A|\leq m_k + \sum_{r=1}^k\frac{m_{r-1}}{2d_r+1}$ for every $A\in\mathcal{F}$ and
\item \label{Fcontains}for every antichain $I$ of $P$, there exists $A\in \mathcal{F}$ such that $I\subseteq A$. 
\end{enumerate}
\end{lem}

\begin{proof}
Apply Lemma~\ref{generalContK} to obtain the functions $f_1,\dots,f_k$. For each antichain $I$, let $T_{1,I},\dots,T_{k,I}$ be disjoint sets as in Lemma~\ref{generalContK} such that $I\subseteq \left(\bigcup_{r=1}^kT_{r,I}\right)\cup f\left(\bigcup_{r=1}^kT_{r,I}\right)$.
 Define
\[\mathcal{F}:=\left\{\left(\bigcup_{r=1}^kT_{r,I}\right)\cup f\left(\bigcup_{r=1}^kT_{r,I}\right): I\subseteq P\text{ is an antichain}\right\}.\]
Then \eqref{setBound} and \eqref{Fcontains} hold by construction. Let us show that \eqref{collectionBound} holds. The number of ways of choosing $T_{1,I}$ is at most $\binom{|P|}{\leq |P|/(2d_1+1)}$. Now, for $2\leq j\leq k$, given that $T_{1,I},\dots,T_{j-1,I}$ have been chosen, we know that $T_{j,I}$ is a subset of $f_{j-1}\left(\bigcup_{r=1}^{j-1}T_{r,I}\right)$. Therefore, there are at most $\binom{m_{j-1}}{\leq m_{j-1}/(2d_{j}+1)}$ ways to choose $T_{j,I}$. This argument proves \eqref{collectionBound}. 
\end{proof}

\section{Applications of the Container-Type Lemma}
\label{countAnti}

\subsection{Counting Antichains}

Before presenting the proofs of Theorems~\ref{countVec} and~\ref{countMulti}, we illustrate the the method by applying Lemma~\ref{containersCor} and Theorem~\ref{booleanThm} to obtain a short proof of a known upper bound on the number of antichains in $\mathcal{P}(n)$. The problem of approximating the number of antichains in $\mathcal{P}(n)$ is known as Dedekind's Problem and has a long history; see Kahn~\cite{KahnAnti}.

\begin{thm}[See, e.g.,~\cite{KahnAnti}]
\label{countBoolean}
The number of antichains in $\mathcal{P}(n)$ is at most \[2^{\left(1+O\left(\sqrt{\log(n)/n}\right)\right)\binom{n}{\left\lceil n/2\right\rceil}}.\] 
\end{thm}

\begin{proof}
We generate a set of containers for the antichains of $\mathcal{P}(n)$ using Lemma~\ref{containersCor}. An upper bound on the number of antichains will then follow by considering all subsets of the containers. Set $d_1:=n\sqrt{\log{n}}$, $d_2:=\sqrt{n\log{n}}$, $m_0:=2^n$, $m_1:=\frac{2\binom{n}{\left\lfloor n/2\right\rfloor}}{1-8d_1/n^2}$ and $m_2:=\frac{\binom{n}{\left\lfloor n/2\right\rfloor}}{1-2d_2/n}$. If $S$ is a subset of $\mathcal{P}(n)$ of cardinality at least $m_1$, then, by Theorem~\ref{booleanThm},
\[\comp(S)\geq \left(|S| - 2\binom{n}{\left\lceil n/2\right\rceil}\right)\left(\frac{n^2}{8}\right)\]
\[ = \left(|S|-m_1\right)\left(\frac{n^2}{8}\right) + \left(m_1 - 2\binom{n}{\left\lceil n/2\right\rceil}\right)\left(\frac{n^2}{8}\right) \geq (|S|-m_1)d_1 + m_1d_1 = |S|d_1\]
since $\left(m_1-2\binom{n}{\left\lfloor n/2\right\rfloor}\right)\left(\frac{n^2}{8}\right) = m_1d_1$.  Similarly, if $S$ has cardinality at least $m_2$, then $\comp(S)\geq |S|d_2$. Therefore, we can apply Lemma~\ref{containersCor} to $\mathcal{P}(n)$ with $k=2$ and these values of $d_1,d_2,m_0,m_1$ and $m_2$ to obtain a collection $\mathcal{F}$ of at most $\binom{2^n}{\leq 2^n/d_1}\binom{m_1}{\leq m_1/d_2}$ containers, each of cardinality at most $m_2+\frac{m_1}{d_2}+\frac{2^n}{d_1}$, such that each antichain of $\mathcal{P}(n)$ is contained in at least one container. Thus, the number of antichains in $\mathcal{P}(n)$ is at most
\begin{align*}\binom{2^n}{\leq 2^n/d_1}\binom{m_1}{\leq m_1/d_2}2^{m_2+\frac{m_1}{d_2} + \frac{2^n}{d_1}}
&\leq \left(e\cdot d_1\right)^{2^n/d_1}\left(e\cdot d_2\right)^{m_1/d_2}2^{m_2+\frac{m_1}{d_2} + \frac{2^n}{d_1}}\\
&= 2^{\frac{2^n\log_2(e\cdot d_1)}{d_1} + \frac{m_1\log_2(e\cdot d_2)}{d_2} + m_2 + \frac{m_1}{d_2} + \frac{2^n}{d_1}} \\
&\leq 2^{m_2 + O\left(\frac{2^n\log{n}}{d_1} + \frac{m_1\log{n}}{d_2}\right)}\end{align*}
for $n$ sufficiently large. Now, bounding the exponent, we get
\begin{align*}m_2 + O\left(\frac{2^n\log{n}}{d_1} + \frac{m_1\log{n}}{d_2}\right)&= \frac{\binom{n}{\left\lfloor n/2\right\rfloor}}{1-2d_2/n} + O\left(\frac{2^n\sqrt{\log{n}}}{n} + \binom{n}{\left\lfloor n/2\right\rfloor}\sqrt{\log{n}/n}\right)\\
&\leq \binom{n}{\left\lfloor n/2\right\rfloor} + \frac{4d_2}{n} \binom{n}{\left\lfloor n/2\right\rfloor} +O\left(\binom{n}{\left\lfloor n/2\right\rfloor}\sqrt{\log{n}/n}\right)\\
& = \left(1+O\left(\sqrt{\log{n}/n}\right)\right)\binom{n}{\left\lfloor n/2\right\rfloor}\end{align*}
since $\binom{n}{\left\lfloor n/2\right\rfloor} =\Theta\left(\frac{2^n}{\sqrt{n}}\right)$ by Stirling's Approximation. The result follows.
\end{proof}

We now prove Theorems~\ref{countVec} and~\ref{countMulti} by following along similar lines to the proof of Theorem~\ref{countBoolean} given above. 

\begin{proof}[Proof of Theorem~\ref{countVec}]
For brevity, define $N:=|\mathcal{V}(q,n)|$ and note that $N=\Theta\left(q^{n^2/2}\right)$ by Theorem~\ref{wilf}. Let $d:=\sqrt{n}q^{n/4}$ and $m:=\frac{\qbinom{n}{\left\lfloor n/2\right\rfloor}{q}}{1-d\qbinom{\left\lceil (n+1)/2\right\rceil}{1}{q}^{-1}}$. Note that Theorem~\ref{wilf} also implies that $m=\Theta\left(q^{n^2/2}\right)$ and so $m$ and $N$ are of the same order of magnitude. By Theorem~\ref{vecSpThm}, if $S$ is a subset of $\Vectors{q}{n}$ of cardinality at least $m$, then 
\[\comp(S)\geq \left(|S|-\qbinom{n}{\left\lfloor n/2\right\rfloor}{q}\right)\qbinom{\left\lceil (n+1)/2\right\rceil}{1}{q}\]
\[=\left(|S|-m\right)\qbinom{\left\lceil (n+1)/2\right\rceil}{1}{q} + \left(m-\qbinom{n}{\left\lfloor n/2\right\rfloor}{q}\right)\qbinom{\left\lceil (n+1)/2\right\rceil}{1}{q}\geq \left(|S|-m\right)d + md = |S|d.\]
Thus, by Lemma~\ref{containersCor}, there is a collection of at most $\binom{N}{\leq N/d}$ containers, each of cardinality at most $m+\frac{N}{d}$, such that each antichain is contained in at least one container. Thus, the number of antichains in $\Vectors{q}{n}$ is at most
\[\binom{N}{\leq N/d}2^{m+\frac{ N}{d}}\leq\left(e\cdot d\right)^{\frac{ N}{d}}2^{m+\frac{ N}{d}}= 2^{m+ \log N + \left(\frac{ N}{d}\right)\left(\log(e\cdot d) + 1\right)}\]
Since $q$ is fixed, we can use Theorem~\ref{wilf} to bound the exponent as follows:
\begin{align*}m+ \log N + \left(\frac{ N}{d}\right)\left(\log(e\cdot d) + 1\right) &\leq\frac{\qbinom{n}{\left\lfloor n/2\right\rfloor}{q}}{1-d\qbinom{\left\lceil (n+1)/2\right\rceil}{1}{q}^{-1}} + O\left(n^2 + \frac{ N\cdot n}{d}\right)\\
&\leq \qbinom{n}{\left\lfloor n/2\right\rfloor}{q} + \frac{2d\qbinom{n}{\left\lfloor n/2\right\rfloor}{q}}{\qbinom{\left\lceil (n+1)/2\right\rceil}{1}{q}} + O\left(\frac{ N\cdot n}{d}\right).\end{align*}
This completes the proof since $ N = \Theta\left(\qbinom{n}{\left\lfloor n/2\right\rfloor}{q}\right)$ and, for fixed $q$, we have $\qbinom{\left\lceil (n+1)/2\right\rceil}{1}{q}=\Theta\left(q^{n/2}\right)$. 
\end{proof}

\begin{proof}[Proof of Theorem~\ref{countMulti}]
Set $d_1:=n\sqrt{\log{n}}$, $d_2:=\sqrt{n\log{n}}$, $m_0:=3^n$, $m_1:=\frac{2\ell_n(n)}{1-50d_1/n^2}$ and $m_2:=\frac{\ell_n(n)}{1-4d_2/n}$. If $S$ is a subset of $\{0,1,2\}^n$ of cardinality at least $m_1$, then, by Theorem~\ref{multisetThm},
\[\comp(S)\geq \left(|S| - 2\ell_n(n)\right)\left(\frac{n^2}{50}\right)\]
\[ = \left(|S|-m_1\right)\left(\frac{n^2}{50}\right) + \left(m_1 - 2\ell_n(n)\right)\left(\frac{n^2}{50}\right) \geq (|S|-m_1)d_1 + m_1d_1 = |S|d_1,\]
since $\left(m_1-2\ell_n(n)\right)\left(\frac{n^2}{50}\right) = m_1d_1$.  Similarly, if the cardinality of $S$ is at least $m_2$, then $\comp(S)\geq |S|d_2$. Therefore, we can apply Lemma~\ref{containersCor} to $\{0,1,2\}^n$ with $k=2$ and these values of $d_1,d_2,m_0,m_1$ and $m_2$ to obtain a collection $\mathcal{F}$ of at most $\binom{3^n}{\leq 3^n/d_1}\binom{m_1}{\leq m_1/d_2}$ containers, each of cardinality at most $m_2+\frac{m_1}{d_2} + \frac{3^n}{d_1}$, such that each antichain of $\{0,1,2\}^n$ is contained in at least one container. We get that the number of antichains in $\{0,1,2\}^n$ is at most
\begin{align*}
\binom{3^n}{\leq 3^n/d_1}\binom{m_1}{\leq m_1/d_2}2^{m_2+\frac{m_1}{d_2} + \frac{3^n}{d_1}}&\leq \left(e\cdot d_1\right)^{3^n/d_1}\left(e\cdot d_2\right)^{m_1/d_2}2^{m_2+\frac{m_1}{d_2} + \frac{3^n}{d_1}}\\
&\leq 2^{m_2 + O\left(\frac{3^n\log{n}}{d_1} + \frac{m_1\log{n}}{d_2}\right)}\end{align*}
for $n$ sufficiently large. Now, bounding the exponent, we get
\begin{align*} m_2 + O\left(\frac{3^n\log{n}}{d_1} + \frac{m_1\log{n}}{d_2}\right)&= \frac{\ell_n(n)}{1-4d_2/n} + O\left(\frac{3^n\log{n}}{d_1} + \frac{\ell_n(n)\log{n}}{d_2}\right)\\
&\leq \ell_n(n) + \frac{8d_2\ell_n(n)}{n}  + O\left(\frac{3^n\log{n}}{d_1} + \frac{\ell_n(n)\log{n}}{d_2}\right)\\
& = \left(1+O\left(\sqrt{\log{n}/n}\right)\right)\ell_n(n)\end{align*}
since $\ell_n(n) =\Theta\left(\frac{3^n}{\sqrt{n}}\right)$ by Theorem~\ref{anderson}. The result follows.
\end{proof}

\begin{rem}
Balogh, Treglown and Wagner~\cite{BaloghWagnerBoolean} have also used the container method to prove an upper bound on the number of antichains in $\mathcal{P}(n)$. Their proof only applies a single-stage container lemma (similar to Lemma~\ref{generalCont}), however, and so the bound that they obtain is weaker than that of Theorem~\ref{countBoolean}. 

In the proofs of Theorems~\ref{countBoolean} and~\ref{countMulti} we used Lemma~\ref{generalContK} with $k=2$ and in the proof of Theorem~\ref{countVec} we used Lemma~\ref{generalContK} with $k=1$. We remark that it is not possible to obtain better bounds by following the same proof with a larger fixed value of $k$.

We should also mention that much stronger bounds than that of Theorem~\ref{countBoolean} are known. In particular, Korshunov~\cite{Korsh} obtained precise asymptotics of the number of antichains in $\mathcal{P}(n)$.
\end{rem}

\subsection{Antichains in Random Subsets}

In this section, we use Lemma~\ref{containersCor} to prove Theorems~\ref{randVecSp} and~\ref{randMultiset}. We remark that our approach is very similar to that of Balogh, Mycroft and Treglown~\cite{RandomSperner} who used a rough form of Theorem~\ref{booleanThm} to show that, when $p\gg 1/n$, the largest antichain in a $p$-random subset of $\mathcal{P}(n)$ is $(1+o(1))p\binom{n}{\left\lfloor n/2\right\rfloor}$ with high probability; the same result was obtained independently by Collares Neto and Morris~\cite{RandomErdos} using a somewhat different approach. 

In both proofs, the first step will be to use Lemma~\ref{generalCont} to obtain a collection of containers for the antichains of the poset. Now, if a $p$-random subset contains a large antichain, then in particular it must contain a large subset of one of the containers (since every antichain is contained in some container). By applying a version of the standard Chernoff Bound (Theorem~\ref{Cher}, stated below) one obtains that the probability that this occurs for any fixed container is very small and, using the fact that the number of containers is not too large, we can apply a simple union bound to complete the proof.

\begin{thm}[The Chernoff Bound; see, e.g.,~\cite{ConcentrationBook}]
\label{Cher}
Suppose that $X_1,\dots,X_n$ are independent random variables taking values in $\{0,1\}$ and let $X:=\sum_{i=1}^nX_i$. Then
\[\mathbb{P}(X>(1+\delta)\mathbb{E}(X))\leq 
\begin{cases}  e^{-\frac{\delta^2\mathbb{E}(X)}{3}}& \text{if }0<\delta<1, \\
						  e^{-\frac{\delta\mathbb{E}(X)}{3}}& \text{if }\delta>1.\end{cases}\]
\end{thm}

We prove Theorem~\ref{randVecSp}.

\begin{proof}[Proof of Theorem~\ref{randVecSp}]
Define $N:=|\Vectors{q}{n}|$ and note that $N=\Theta\left(\qbinom{n}{\left\lceil n/2\right\rceil}{q}\right)$ by Theorem~\ref{wilf}. Let $p=c(\varepsilon,q)/q^{n/2}$ for some constant $c(\varepsilon,q)$, to be chosen later, and let $A_p$ be a subset of $\Vectors{q}{n}$ obtained by including each element independently with probability $p$. Let $m:=(1+\varepsilon/2)\qbinom{n}{\left\lceil n/2\right\rceil}{q}$. By Theorem~\ref{vecSpThm}, every subset $S$ of $\Vectors{q}{n}$ of cardinality at least $m$ satisfies
\[\comp(S)\geq \left(1-\frac{1}{1+\varepsilon/2}\right)\qbinom{\left\lceil (n+1)/2\right\rceil}{1}{q}|S|\geq \left(\frac{\varepsilon}{3}\right)\qbinom{\left\lceil (n+1)/2\right\rceil}{1}{q}|S|\]
given that $\varepsilon\in(0,1)$. So, we set $d:=\left(\frac{\varepsilon}{3}\right)\qbinom{\left\lceil (n+1)/2\right\rceil}{1}{q}$ and let $f$ be the function resulting from applying Lemma~\ref{generalCont} to $\Vectors{q}{n}$ with these values $m$ and $d$. Note that, since $q$ is constant, we have $d=\Theta\left(q^{n/2}\right)$. In particular, $p$ is bounded below by a constant multiple of $c(\varepsilon,q)/d$.

 Fix $T\subseteq \Vectors{q}{n}$ such that $|T|\leq\frac{ N}{2d+1}$. The probability that $A_p$ contains an antichain $I$ of size at least $(1+\varepsilon)p\qbinom{n}{\left\lceil n/2\right\rceil}{q}$ with $T\subseteq I\subseteq T\cup f(T)$ is at most
\[\mathbb{P}(T\subseteq A_p)\cdot\mathbb{P}\left(|A_p\cap f(T)|\geq p(1+\varepsilon)\qbinom{n}{\left\lceil n/2\right\rceil}{q}-|T|\right)\]
since $T\cap f(T)=\emptyset$. By Theorem~\ref{wilf}, $|T|\leq \frac{ N}{2d+1}=O\left(q^{-n/2}\qbinom{n}{\left\lceil n/2\right\rceil}{q}\right)$. By choosing $c(\varepsilon,q)$ sufficiently large, we can make $|T|\leq \left(\frac{\varepsilon}{4}\right)p\qbinom{n}{\left\lceil n/2\right\rceil}{q}$ and so we get that the above probability is at most
\[p^{|T|}\mathbb{P}\left(|A_p\cap f(T)|\geq p(1+3\varepsilon/4)\qbinom{n}{\left\lceil n/2\right\rceil}{q}\right).\]
Applying Theorem~\ref{Cher} with $\delta=\frac{(1+3\varepsilon/4)\qbinom{n}{\left\lceil n/2\right\rceil}{q}- |f(T)|}{|f(T)|}$ we get that
\[\mathbb{P}\left(|A_p\cap f(T)|\geq p(1+3\varepsilon/4)\qbinom{n}{\left\lceil n/2\right\rceil}{q}\right)\leq 
\begin{cases}  e^{-\frac{\delta^2|f(T)|p}{3}}& \text{if }0<\delta<1, \\
						  e^{-\frac{\delta|f(T)|p}{3}}& \text{if }\delta>1.\end{cases}\]
Since $|f(T)|\leq m = (1+\varepsilon/2)\qbinom{n}{\left\lceil n/2\right\rceil}{q}$, the right side is bounded above by $e^{-\varepsilon^2 mp/100}$. Therefore, the probability that $A_p$ contains an antichain of size at least $(1+\varepsilon)p\qbinom{n}{\left\lceil n/2\right\rceil}{q}$ is at most
\[\sum_{\substack{T\subseteq \Vectors{q}{n}\\ |T|\leq  N/(2d+1)}}p^{|T|}e^{-\varepsilon^2mp/100} \leq\sum_{a=0}^{ N/(2 d+1)}\binom{ N}{a}p^ae^{-\varepsilon^2mp/100}.\]
Given that $c(\varepsilon,q)$ is large enough, the quantity $\binom{ N}{a}p^a$ is increasing for $0\leq a\leq \frac{ N}{2 d+1}$. Thus, this is bounded above by
\[\left(\frac{ N}{2 d+1} + 1\right)\left(e(2d+1)p\right)^{ N/(2 d+1)}e^{-\varepsilon^2mp/100}.\]
Since $d=\Theta\left(q^{n/2}\right)$ and $N=\Theta(m)$, the quantity $e(2d+1)p$ bounded above by a constant multiple of $c(\varepsilon,q)$ and $\varepsilon^2mp/100$ is bounded below by a constant multiple (depending only on $\varepsilon$ and $q$) of $c(\varepsilon,q)\left(\frac{N}{2d+1}\right)$. Therefore, we may choose $c(\varepsilon,q)$ large enough so that $\log(e(2d+1)p)\left(\frac{N}{2d+1}\right)< \varepsilon^2mp/200$, and so the above expression is $o(1)$. The result follows.
\end{proof}

Let us note that Theorem~\ref{randVecSp} is best possible up to the choice of the constant $c(\varepsilon,q)$. Let $A_p$ be a $p$-random subset of $\Vectors{q}{n}$ where $p=c/q^{n/2}$ for some constant $c>0$ and let $n$ be large with respect to $q$ and $c$. Given $x\in\Vectors{q}{n}$ of dimension $\left\lceil (n+1)/2\right\rceil$, the expected number of subspaces of $x$ of dimension $\left\lfloor n/2\right\rfloor$ contained in $A_p$ is constant. Thus, with probability bounded away from zero, a constant proportion of the elements of $\Vectors{q}{n}\cap A_p$ of dimension $\left\lceil (n+1)/2\right\rceil$ have no subspaces of dimension $\left\lfloor n/2\right\rfloor$ contained in $A_p$. Taking this set along with all subspaces of $\mathbb{F}_q^n$ of dimension $\left \lfloor n/2\right\rfloor$ contained in $A_p$ gives us an antichain of size at least $(1+\varepsilon)p\qbinom{n}{\left\lfloor n/2\right\rfloor}{q}$ for some $\varepsilon=\varepsilon(q,c)>0$. Here, we used the fact that $\qbinom{n}{\left\lceil (n+1)/2\right\rceil}{q}=\Theta\left(\qbinom{n}{\left\lfloor n/2\right\rfloor}{q}\right)$ for fixed $q$ (Theorem~\ref{wilf}). A similar argument shows that Theorem~\ref{randMultiset} is best possible up to the choice of $c(\varepsilon)$. Next, we apply Lemma~\ref{generalCont} and Theorem~\ref{multisetThm} to prove Theorem~\ref{randMultiset}.

\begin{proof}[Proof of Theorem~\ref{randMultiset}]
Let $p=c(\varepsilon)/n$ for some constant $c(\varepsilon)$ to be chosen later and let $A_p$ be a subset of $\{0,1,2\}^n$ obtained by including each element independently with probability $p$. Let $m_1:=3\ell_n(n)$ and let $m_2:=(1+\varepsilon/2)\ell_n(n)$. By Theorem~\ref{multisetThm}, every subset $S$ of $\{0,1,2\}^n$ of cardinality at least $m_1$ satisfies
\[\comp(S)\geq \left(|S|-2\ell_n(n)\right)\left(\frac{n^2}{50}\right)\geq \left(\frac{n^2}{200}\right)|S|\]
Also, by Theorem~\ref{multisetThm}, every subset $S$ of $\{0,1,2\}^n$ of cardinality at least $m_2$ satisfies 
\[\comp(S)\geq \left(|S| - \ell_n(n)\right)\left(\frac{n}{4}\right)\geq  \left(\frac{\varepsilon n}{12}\right) |S|\]
given that $\varepsilon\in(0,1)$.  So, we set $d_1:=\frac{n^2}{200}$ and $d_2:= \frac{\varepsilon n}{12}$. Let $f_1$ and $f_2$ be the functions resulting from applying Lemma~\ref{generalContK} to $\{0,1,2\}^n$ with $k=2$ and the values of $d_1,d_2,m_1,m_2$ above. 

Fix $T_1\subseteq\{0,1,2\}^n$ such that $|T_1|\leq \frac{3^n}{2d_1+1}$ and fix $T_2\subseteq f_1(T_1)$ such that $|T_2|\leq \frac{m_1}{2d_2+1}$. The probability that $A_p$ contains an antichain $I$ of size at least $(1+\varepsilon)p\ell_n(n)$ with $T_1\cup T_2\subseteq I\subseteq T_1\cup T_2\cup f_2(T_1\cup T_2)$ is at most
\begin{equation}\label{probthing}\mathbb{P}(T_1\cup T_2\subseteq A_p)\cdot\mathbb{P}\left(\left|A_p\cap f_2\left(T_1\cup T_2\right)\right|\geq p(1+\varepsilon)\ell_n(n) - |T_1\cup T_2|\right)\end{equation}
since $\left(T_1\cup T_2\right)\cap f_2\left(T_1\cup T_2\right)=\emptyset$. Note that $|T_1\cup T_2|\leq \frac{100\cdot3^n}{n^2} + \frac{6\ell_n(n)}{\varepsilon n} \leq \frac{10\ell_n(n)}{\varepsilon n}$ for large $n$ by Theorem~\ref{anderson}. Since $p=c(\varepsilon)/n$, we may choose $c(\varepsilon)$ sufficiently large so that $|T_1\cup T_2|\leq \left(\frac{\varepsilon}{4}\right)p\ell_n(n)$. Therefore, the expression in \eqref{probthing} is bounded above by
\[p^{|T_1\cup T_2|}\mathbb{P}\left(|A_p\cap f_2(T_1\cup T_2)|\geq p(1+3\varepsilon/4)\ell_n(n)\right).\]
Applying Theorem~\ref{Cher} with $\delta=\frac{(1+3\varepsilon/4)\ell_n(n) - |f_2(T_1\cup T_2)|}{|f_2(T_1\cup T_2)|}$ we get that
\[\mathbb{P}\left(|A_p\cap f_2(T_1\cup T_2)|\geq p(1+3\varepsilon/4)\ell_n(n)\right)\leq \begin{cases}  e^{-\frac{\delta^2\left|f_2(T_1\cup T_2)\right|p}{3}}& \text{if }0<\delta<1, \\
						  e^{-\frac{\delta\left|f_2(T_1\cup T_2)\right|p}{3}}& \text{if }\delta>1.\end{cases}\]
Since $\left|f_2\left(T_1\cup T_2\right)\right|\leq (1+\varepsilon/2)\ell_n(n)$, the right side is bounded above by $e^{-\varepsilon^2\ell_n(n)p/100}$. Therefore, the probability that $A_p$ contains an antichain of size at least $(1+\varepsilon)p\ell_n(n)$ is at most
\[\sum_{a=0}^{\frac{3^n}{2d_1+1}}\sum_{b=0}^{\frac{m_1}{2d_2+1}}\binom{3^n}{a}\binom{m_1}{b}p^{a+b}e^{-\varepsilon^2\ell_n(n)p/100}\]
For $c(\varepsilon)$ sufficiently large, the quantity $\binom{3^n}{a}\binom{m_1}{b}p^{a+b}$ is increasing in both $a$ and $b$ in the range $0\leq a\leq \frac{3^n}{2d_1+1}$ and $0\leq b\leq \frac{m_1}{2d_2+1}$. Therefore, the above expression is at most
\[ \left(\frac{3^n}{2d_1+1}\right)\left(\frac{m_1}{2d_2+1}\right)\left(3ed_1p\right)^{\frac{3^n}{2d_1+1}}\left(3ed_2p\right)^{\frac{m_1}{2d_2+1}}e^{-\varepsilon^2\ell_n(n)p/100}.\]
By definition of $d_2$ and $p$, we have that $3ed_2p$ is bounded above by a constant multiple of $c(\varepsilon)$ and, since $\ell_n(n) = \Theta\left(m_1\right)$, we have that $\varepsilon^2\ell_n(n)p/100$ is bounded below by a constant multiple (depending only on $\varepsilon$) of $c(\varepsilon)\left(\frac{m_1}{2d_2+1}\right)$. Therefore we may choose $c(\varepsilon)$ large enough so that $\log(3ed_2p)\left(\frac{m_1}{2d_2+1}\right)< \varepsilon^2\ell_n(n)p/200$ and so the above expression is $o(1)$. This completes the proof. 
\end{proof}

To close this section, we remark that none of the proofs in this section required asymptotically sharp supersaturation results. In applying the container method, it is often sufficient to have a supersaturation result which is best possible up to a constant factor (such as the graph supersaturation theorem of Erd\H{o}s and Simonovits~\cite{Hyper}).

\section{Open Problems}
\label{concl}

For general $r\geq1$ one can define a partial order on $\{0,\dots,r\}^n$ where $x\leq y$ if and only if $x_j\leq y_j$ for all $j$. For $0\leq i\leq rn$ let $\ell_i(n,r)$ be the number of elements of $\{0,\dots,r\}^n$ with coordinate sum equal to $i$. We propose the following conjecture. 

\begin{conj}
\label{middleLevels}
Let $r,n$ be positive integers such that $n$ is large with respect to $r$. Then, for $1\leq m\leq (r+1)^n$, the number of comparable pairs in a subset of $\{0,\dots,r\}^n$ of cardinality $m$ is minimised by a set whose elements have coordinate sum as close to $rn/2$ as possible. 
\end{conj}

We remark that the same conjecture was made independently by Balogh and Wagner~\cite{BaloghWagnerKleitman}. The case $r=1$ of Conjecture~\ref{middleLevels} follows from the theorem of Kleitman~\cite{superKleitman} (without the requirement that $n$ is large), but it is open for all other values of $r$. In particular, the case $r=2$ of Conjecture~\ref{middleLevels} is strictly stronger than Theorem~\ref{multisetThm}. The following weak form of Conjecture~\ref{middleLevels} would be good enough for certain applications.

\begin{conj}
[Weak form of Conjecture~\ref{middleLevels}]
Let $r,k$ be positive integers. There exist a positive constant $c(r,k)$ such that for every $S\subseteq \{0,\dots,r\}^n$ we have
\[\comp(S)\geq \left(|S|-k\ell_{\left\lceil rn/2\right\rceil}(n,r)\right)c(r,k)n^k\]
\end{conj}

The poset $\{0,\dots,r\}^n$ can also be viewed as the poset of divisors of the $r$th power of a square-free integer with $n$ prime factors. Another natural direction for future work could be to prove a non-trivial supersaturation theorem for the poset of divisors of a general integer (perhaps with the number of distinct prime factors sufficiently large with respect to the largest multiplicity of a prime factor). As mentioned earlier, Anderson~\cite{Andersondivisors} proved that such posets have regular coverings by chains. It seems possible that one could prove a supersaturation theorem for such a poset by considering a chain chosen randomly according to a distribution based on a regular covering by chains and applying Lemma~\ref{randomCount}. The devil, however, is in the details.

Also, one could consider supersaturation problems for other classical posets. In particular, it would also be interesting to prove supersaturation results for posets which do not have regular coverings by chains, or for posets which are not even ranked. 

In another direction, one could try to improve our bounds on the number of antichains in $\Vectors{q}{n}$ and $\{0,1,2\}^n$ (Theorems~\ref{countVec} and~\ref{countMulti}) and obtain similar bounds for other posets. 

\begin{noteAdded}
Following the submission of this paper, Balogh, Pet\v{r}\'{i}\v{c}ov\'{a} and Wagner~\cite{BPW} disproved Conjecture~\ref{middleLevels} for all $r\geq2$ and Samotij~\cite{SamotijKleit} proved the conjecture of Kleitman~\cite{superKleitman} that, for any $k\geq2$ and $1\leq m\leq 2^n$, the number of chains of length $k$ in a subset of $\mathcal{P}(n)$ of size $m$ is minimised by a collection of subsets of $[n]$ of size as close to $n/2$ as possible. 
\end{noteAdded}

  \bibliographystyle{plain}

  \appendix
  \section{Rank Numbers in \texorpdfstring{$\boldsymbol{\Vectors{q}{n}}$}{V(q,n)} and \texorpdfstring{$\boldsymbol{\{0,1,2\}^n}$}{{0,1,2}\textasciicircum n}}
  
We provide a number of results regarding the rank numbers in $\Vectors{q}{n}$ and $\{0,1,2\}^n$ that we require in the paper. The first is a theorem of Wilf~\cite{Wilf} which is used in the proofs of Theorems~\ref{countVec} and~\ref{randVecSp}.
  
\begin{thm}[Wilf~\cite{Wilf}]
\label{wilf}
For a fixed prime power $q$ and $n\to\infty$, we have
\[|\Vectors{q}{n}|=\Theta\left(q^{n^2/2}\right),\text{ and}\]
\[\qbinom{n}{\left\lfloor n/2\right\rfloor}{q}=\Theta\left(q^{n^2/2}\right).\]
In particular, $|\Vectors{q}{n}|=\Theta\left(\qbinom{n}{\left\lfloor n/2\right\rfloor}{q}\right)$.
\end{thm}

The following result of Anderson~\cite{AndersonVariance} is used in the proofs of Theorems~\ref{countMulti} and~\ref{randMultiset}.

\begin{thm}[Anderson~\cite{AndersonVariance}]
\label{anderson}
$\ell_n(n)=\Theta\left(\frac{3^n}{\sqrt{n}}\right)$.
\end{thm}

A sequence $a_0,\dots,a_N$ is said to be \emph{log-concave} if $\frac{a_{i+1}}{a_i}\leq \frac{a_i}{a_{i-1}}$ for $1\leq i\leq N-1$. In the proof of Theorem~\ref{multisetThm} we use the following theorem, which has been proven in several different papers (see Chapter 4 of~\cite{AndersonBook} for references). To our knowledge, the earliest proof is due to Anderson~\cite{Andersondivisors}.

\begin{thm}[Anderson~\cite{Andersondivisors}]
\label{logConcave}
The sequence $\ell_0(n),\dots,\ell_{2n}(n)$ is log-concave. 
\end{thm}

We now provide some quantitative upper and lower bounds on the ratios $\ell_{i+1}(n)/\ell_i(n)$ which we used in Section~\ref{multisetSec}. We obtain different bounds depending on whether $1\leq i\leq n-3$, $i=n-2$ or $i=n-1$.

\begin{lem}
\label{ratio1}
For $n\geq 4$ and $1\leq i\leq n-3$,
\[\frac{n+1}{i+1}\leq \frac{\ell_{i+1}(n)}{\ell_i(n)}\leq \frac{\left\lceil i/2\right\rceil - \left\lfloor i/2\right\rfloor}{\left \lceil i/2\right\rceil} + \frac{n-\left\lceil i/2\right\rceil}{\left\lceil (i+2)/2\right\rceil}.\]
\end{lem}

\begin{proof}
Begin by giving each element $y$ of $L_{i+1}$ a `weight' of $1$ and then redistributing the weight of $y$ evenly among the $x\in L_i$ with $x<y$. For $0\leq s\leq \left\lfloor i/2\right\rfloor$, the weight recieved by $x\in L_i^s$ is precisely
\[f(s):=\frac{i-2s}{i-s} + \frac{n-i+s}{i-s+1}.\]
Note that $f(s)$ is well defined as $i\geq1$ and so $s < i$. We claim that $f(s)$ is non-decreasing for $s$ in the range $0\leq s\leq \left\lfloor i/2\right\rfloor$. Indeed, for $i\geq2$ and $0\leq s\leq \left\lfloor (i-2)/2\right\rfloor$, we have
\[f(s+1)-f(s) = \left(\frac{i-2s-2}{i-s-1} + \frac{n-i+s+1}{i-s}\right) - \left(\frac{i-2s}{i-s} + \frac{n-i+s}{i-s+1}\right)\]
\[=\frac{s(-n +i - 1) + ni - n -i^2-1}{(i-s+1)(i-s)(i-s-1)}.\]
Since $n>i$ and $s\leq \frac{i-2}{2}$, this expression is bounded below by
\[\frac{(i-2)(-n+i-1) + 2ni-2n-2i^2-2}{2(i-s+1)(i-s)(i-s-1)} = \frac{i(n-i-3)}{2(i-s+1)(i-s)(i-s-1)}\]
which is non-negative since $i\leq n-3$. Therefore, the total weight is bounded below by $|L_i|$ multiplied by the weight recieved by an element of $L_i^0$ and bounded above by $|L_i|$ multiplied by the weight recieved by an element of $L_i^{\left\lfloor i/2\right\rfloor}$. However, clearly, the total weight is precisely $|L_{i+1}|$ and so we obtain
\[\left(\frac{n+1}{i+1}\right)|L_i|\leq |L_{i+1}| \leq \left(\frac{\left\lceil i/2\right\rceil - \left\lfloor i/2\right\rfloor}{\left \lceil i/2\right\rceil} + \frac{n-\left\lceil i/2\right\rceil}{\left\lceil (i+2)/2\right\rceil}\right)|L_i|\]
as desired. 
\end{proof}

\begin{lem}
\label{ration-2}
For $n\geq5$ we have 
\[\frac{n+2}{n}\leq \frac{\ell_{n-1}(n)}{\ell_{n-2}(n)}\leq \frac{4n+7}{4n-2}\]	
\end{lem}

\begin{proof}
Distribute weights to the elements of $L_i$ in the same way as in the proof of Lemma~\ref{ratio1} and define $f(s)$ in the same way as well, except that $i=n-2$. As in that proof, for $0\leq s\leq \left\lfloor (n-4)/2\right\rfloor$, we have
\[f(s+1)-f(s) =\frac{s(-n + (n-2) - 1) + n(n-2) - n -(n-2)^2-1}{((n-2)-s+1)((n-2)-s)((n-2)-s-1)}\]
\[= \frac{-3s + n-5}{(n-1-s)(n-2-s)(n-3-s)}\]
which is non-negative if and only if $s\leq \frac{n-5}{3}$. So, we have that $f(s)$ is non-decreasing in the range $0\leq s\leq \left\lfloor\frac{n-2}{3}\right\rfloor$ and non-increasing in the range $\left\lfloor\frac{n-2}{3}\right\rfloor\leq s\leq\frac{n-2}{2}$. It follows that
\[\min\left\{f(0),f\left(\frac{n-2}{2}\right)\right\}|L_i|\leq |L_{i+1}| \leq f\left(\frac{n-2}{3}\right)|L_i|\]
and the result follows since $f\left(\frac{n-2}{2}\right) = \frac{n+2}{n} <\frac{n+1}{n-1} = f(0)$ and $f\left(\frac{n-2}{3}\right) = \frac{4n+7}{4n-2}$. 
\end{proof}

\begin{lem}
\label{ration-1}
For $n\geq5$ we have 
\[\frac{\left\lceil (n-1)/2\right\rceil - \left\lfloor (n-1)/2\right\rfloor}{\left \lceil (n-1)/2\right\rceil} + \frac{\left\lfloor (n+1)/2\right\rfloor}{\left\lceil (n+1)/2\right\rceil}\leq \frac{\ell_{n}(n)}{\ell_{n-1}(n)}\leq \frac{n+1}{n}\]	
\end{lem}

\begin{proof}
Proceed as in the previous two proofs. This time, we have 
\[f(s+1)-f(s) =\frac{s(-n + (n-1) - 1) + n(n-1) - n -(n-1)^2-1}{((n-1)-s+1)((n-1)-s)((n-1)-s-1)}\]
\[=\frac{-2(s+1)}{(n-s)(n-1-s)(n-2-s)}\]
and so $f(s)$ is decreasing in $s$ for $0\leq s \leq \left\lfloor \frac{n-1}{2}\right\rfloor$. Therefore,
\[f\left(\left\lfloor \frac{n-1}{2}\right\rfloor\right)|L_i|\leq |L_{i+1}|\leq f(0)|L_i|\]
and we are done since $f\left(\left\lfloor \frac{n-1}{2}\right\rfloor\right) = \frac{\left\lceil (n-1)/2\right\rceil - \left\lfloor (n-1)/2\right\rfloor}{\left \lceil (n-1)/2\right\rceil} + \frac{\left\lfloor (n+1)/2\right\rfloor}{\left\lceil (n+1)/2\right\rceil}$ and $f(0) = \frac{n+1}{n}$. 
\end{proof}

\end{document}